\documentclass[a4paper,oneside,11pt]{article} 

\usepackage[bottom]{footmisc}

\usepackage{a4wide}

\usepackage{amssymb}
\usepackage{tikz}
\usepackage{amsmath}

\usepackage{comment}

\usepackage{nameref}

\usepackage{bbm}

\usepackage{enumitem, hyperref}
\makeatletter
\def\namedlabel#1#2{\begingroup
    #2%
    \def\@currentlabel{#2}%
    \phantomsection\label{#1}\endgroup
}
\makeatother

\usepackage[all,cmtip]{xy}
\usepackage{tikz-cd}
\usepackage{graphicx}

\usepackage[all]{xy}

\usepackage{mathtools}
\DeclarePairedDelimiter\ev{\langle}{\rangle}

\usepackage[utf8]{inputenc} 
\usepackage[T1]{fontenc} 
\usepackage{textcomp} 
\usepackage[english]{babel}
\usepackage[autolanguage]{numprint} 
\usepackage{hyperref} 
\usepackage{graphicx} 
\usepackage{verbatim} 

\usepackage{amsmath,amssymb,amsthm} 
\usepackage{comment}

\usepackage{fancyhdr} 
\renewcommand\[{\begin{equation}}\renewcommand\]{\end{equation}} 
\renewcommand\epsilon\varepsilon 
\renewcommand\phi\varphi 

\newtheorem{innercustomthm}{Theorem}
\newenvironment{customthm}[1]
  {\renewcommand\theinnercustomthm{#1}\innercustomthm}
  {\endinnercustomthm}

\newcommand\NNN{\mathbb{N}} 
\newcommand\NN{\mathcal{N}}
\newcommand\TT{\mathcal{T}}
\newcommand\ZZ{\mathbb{Z}} 
\newcommand\QQ{\mathbb{Q}} 
\newcommand\ab\allowbreak 

\newcommand\GW{\operatorname{GW}}

\newcommand\Spec{\operatorname{Spec}}
\newcommand\M{M_{-1}}
\newcommand\HIfr{\mathbf{HI}^{\operatorname{fr}}}
\newcommand\HIgtr{\mathbf{HI}^{\operatorname{gtr}}}
\newcommand\MQ{M_{-1,\QQ}}
\newcommand\HIMW{\mathbf{HI}^{\operatorname{MW}}}
\newcommand\Hp{\mathbf{H}_{\bullet}}
\newcommand\HI{\mathbf{HI}}
\newcommand\SHS{\mathbf{SH}^{S^1}}
\newcommand\DA{\mathbf{D}_{\AAA^1}}
\newcommand\DAeff{\mathbf{D}^{\operatorname{eff}}_{\AAA^1}}

\newcommand\colim{\operatorname{colim}}

\newcommand\CHt{\widetilde{\operatorname{CH}}}

\newcommand\Hominter{\underline{\mathcal{H}om}}
\newcommand\Hom{\operatorname{Hom}}

\newcommand\Tr{\operatorname{Tr}}
\newcommand\tr{\operatorname{tr}}

\newcommand\res{\operatorname{res}}

\newcommand\Id{\operatorname{Id}}

\newcommand\FFF{\mathcal{F}}

\newcommand\HM{\bold{HM}}

\newcommand\Sm{\operatorname{Sm}_k}

\newcommand\Cortilde{\widetilde{\operatorname{Cor}}}

\newcommand\eeta{\boldsymbol{\eta}}
\newcommand\Om{\operatorname{\Omega}}
\newcommand\Ab{\mathcal{A}b}

\newcommand\KMW{\underline{\operatorname{K}}^{MW}}

\newcommand\AAA{\mathbb{A}}

\providecommand{\keywords}[1]
{
  \small	
  \textbf{\textit{Keywords---}} #1
}

\providecommand{\Codes}[1]
{
  \small	
  \textbf{\textit{MSC---}} #1
}

\newcommand\kMW{\mathbf{K}^{\text{MW}}}

\newcommand\OO{\mathcal{O}}
\newcommand\LL{\mathcal{L}}

\newcommand\LLL{\mathfrak{L}}

\newcommand\Gm{\mathbb{G}_m}
\newcommand\SH{\mathbf{SH}}
\newcommand\PP{\mathbb{P}}

\newcommand\Grp{\mathbf{Grp}}

\theoremstyle{definition} 
\newtheorem{Def}{Definition}[subsection] 
\theoremstyle{plain} 
\newtheorem{Pro}[Def]{Proposition} 
\newtheorem{Lem}[Def]{Lemma} 
\newtheorem{The}[Def]{Theorem} 
\newtheorem{Conj}[Def]{Conjecture} 
\newtheorem{Claim}[Def]{Claim} 
\newtheorem{Cor}[Def]{Corollary} 
\theoremstyle{remark} 
\newtheorem{Exe}[Def]{Example} 
\newtheorem{Rem}[Def]{Remark} 
\newtheorem{Par}[Def]{} 

\title{MW-homotopy sheaves and Morel generalized transfers} 
\author{\sc Niels Feld\footnote{Adress: Institut Fourier, 100 Rue des Mathématiques, Grenoble, France.}
\footnote{E-mail adress: <niels.feld@univ-grenoble-alpes.fr>.}  } 

\date{2020} 

\begin{document} 

\maketitle 


\begin{abstract} We explore a conjecture of Morel about the Bass-Tate transfers defined on the contraction of a homotopy sheaf \cite{Mor12} and prove that the conjecture is true with rational coefficients. Moreover, we study the relations between (contracted) homotopy sheaves, sheaves with Morel generalized transfers and MW-homotopy sheaves, and prove an equivalence of categories. As applications, we describe the essential image of the canonical functor that forgets MW-transfers and use theses results to discuss the {\it conservativity conjecture} in $\AAA^1$-homotopy due to Bachmann and Yakerson \cite[Conjecture 1.1]{BachmannYakerson18}.
\end{abstract}


\Codes{14C17, 14C35, 11E81}
\tableofcontents

\section{Introduction}
\subsection{Current work}

\par In \cite{Mor12}, Morel studied homotopy invariant Nisnevich sheaves in order to provide computational tools in $\AAA^1$-homotopy analogous to Voevodsky's theory of sheaves with transfers. The most basic result is that (unramified) sheaves are characterized by their sections on fields and some extra data (see Subsection \ref{SubsectionUnramifiedSheaves}). One of the main theorem of \cite{Mor12} is the equivalence between the notions of {\em strongly $\AAA^1$-invariant} and {\em strictly $\AAA^1$-invariant} for sheaves of abelian groups (see {\em loc. cit.} Theorem 1.16). In order to prove this, Morel defined geometric transfers on the contraction $\M$ of a homotopy sheaf (i.e. a strongly $\AAA^1$-invariant Nisnevich sheaf of abelian groups). The definition is an adaptation of the original one of Bass and Tate for Milnor K-theory \cite{BassTate73}. Morel proved that the transfers were functorial (i.e. they do not depend on the choice of generators) for any two-fold contraction $M_{-2}$ of a homotopy sheaf and conjectured that the result should hold for $\M$ (see Conjecture \ref{ConjectureMoinsUn} or \cite[Remark 4.31]{Mor12}).

\par The notion of {\em sheaves with generalized transfers} was first defined in \cite[Definition 5.7]{Mor11} as a way to formalize the different structures naturally arising on some homotopy sheaves. In Section \ref{SectionSheavesWithGeneralizedTransfers}, we give a slightly modified definition of {\em sheaves with generalized transfers} which takes into account twists by the usual line bundles. Following \cite[Chapter 5]{Mor12}, we define the Rost-Schmid complex associated to such homotopy sheaves and study the usual pushforward maps $f_*$, pullback maps $g^*$, $\GW$-action $\ev{a}$ and residue maps $\partial$. Moreover, we prove the following theorem.
\begin{customthm}{1} [see Theorem \ref{MWModuleFM}]
 
 Let $M\in \HIgtr(k)$ be a homotopy sheaf with generalized transfers. The presheaf $\tilde{\Gamma}_*(M)$ of abelian groups, defined by
 
 \begin{center}
 $\tilde{\Gamma}_*(M)(X)=A^0(X,M\otimes(\LLL_{X/k})^{\vee})$
 
 \end{center}
 for any smooth scheme $X/k$, is a MW-homotopy sheaf canonically isomorphic to $M$ as presheaves. 
 \end{customthm}
In Section \ref{SectionMorelConjecture}, we recall the construction of the Bass-Tate transfer maps on a contracted homotopy sheaf $\M$ and prove that this defines a structure of generalized transfers: 
\begin{customthm}{2}[see Theorem \ref{HIMoinsUntogtr}]
Let $M\in \HI(k)$ be a homotopy sheaf. Then:
\begin{enumerate}
\item Assume that $2$ is invertible. The rational contracted homotopy sheaf $M_{-1,\QQ}$ is a homotopy sheaf with generalized transfers.
\item Assuming Conjecture \ref{ConjectureMoinsUn}, the contracted homotopy sheaf $\M$ is a homotopy sheaf with generalized transfers.
\end{enumerate}
\end{customthm}
In particular, we obtain the following intersection multiplicity formula which was left open in \cite{Fel18}:
\begin{customthm}{3}[see Theorem \ref{R1c_fort}]
Let $M\in \HI(k)$ be a homotopy sheaf. Consider $\phi:E\to F$ and $\psi:E\to L$ with $\phi$ finite. Let $R$ be the ring $F\otimes_E L$. For each $\mathfrak{p}\in \Spec R$, let $\phi_\mathfrak{p}:L\to R/\mathfrak{p}$ and $\psi_\mathfrak{p}:F\to R/\mathfrak{p}$ be the morphisms induced by $\phi$ and $\psi$. One has
\begin{center}
$\M(\psi)\circ \Tr_{\phi}=\displaystyle \sum_{\mathfrak{p}\in \Spec R}e_{\mathfrak{p},\epsilon} \Tr_{\phi_\mathfrak{p}}\circ \M(\psi_\mathfrak{p})$
\end{center}
where $e_{\mathfrak{p},\epsilon}=\sum_{i=1}^{e_\mathfrak{p}}\ev{-1}^{i-1}$ is the quadratic form associated to the length $e_{\mathfrak{p}}$ of the localized ring $R_{(\mathfrak{p})}$.

\end{customthm}

Generalizing ideas of Voevodsky, Calmès and Fasel introduced the additive symmetric monoidal category $\Cortilde_k$ of smooth $k$-schemes with morphisms given by the so-called {\em finite Milnor-Witt correspondences} (see \cite[Chapter 2]{BCDFO}). In Section \ref{SectionMWHomotopySheaves}, we recall the basic definitions regarding this theory and prove that any homotopy sheaf with MW-transfers has a structure of sheaf with generalized transfers. More precisely, we show that the two notions coincides:

\begin{customthm}{4}[Theorem \ref{EquivalenceHIMWandHIgtr}] There is a pair of functors
\begin{center}
$\xymatrix{
\HIMW(k)
\ar@<1ex>[r]^{\tilde{\Gamma}^*}
&
\HIgtr(k)
\ar@<1ex>[l]^{\tilde{\Gamma}^*}
}$
\end{center}
that forms an equivalence between the category of homotopy sheaves with MW-transfers and the category of homotopy sheaves with generalized transfers.
\end{customthm}
In Section \ref{SectionApplications}, we prove the following theorem that characterize the essential image of the functor ${\tilde{\gamma}_*:\HIMW(k)\to \HI(k)}$ that forgets MW-transfers. 

\begin{customthm}{5}[Theorem \ref{EssImageTransfers}]
Let $M\in \HI(k)$ be a homotopy sheaf. The following assertions are equivalent:
\begin{enumerate}
\item[(i)]  There exists $M'\in \HI(k)$ satisfying Conjecture
 \ref{ConjectureMoinsUn} and such that $M\simeq M'_{-1}$.
\item[(ii)] There exists a structure of generalized transfers on $M$.
\item[(iii)] There exists a structure of MW-transfers on $M$.
\item[(iv)] There exists $M''\in \HI(k)$ such that $M\simeq M''_{-2}$. 
\end{enumerate}
\end{customthm}

This result is linked with the conservativity conjecture from \cite{BachmannYakerson18} and allows us to prove the following theorems.

\begin{customthm}{6}[Corollary \ref{RationalBYconjecture}]Let $d>0$ be a natural number. The Bachmann-Yakerson conjecture holds (integrally) for $d=2$ and rationally for $d=1$: namely, the canonical functor
 \begin{center}
$\SHS(k)(2) \to \SH(k)$
\end{center}
is conservative on bounded below objects, the canonical functor
 \begin{center}
$\SHS(k)(1) \to \SH(k)$
\end{center}
is conservative on rational bounded below objects, and the canonical functor
\begin{center}
$\HI(k,\QQ)(1)\to \HIfr(k,\QQ)$
\end{center}
is an equivalence of abelian categories.
\par Moreover, let $\mathcal{X}$ be a pointed motivic space. Then the canonical map
\begin{center}
$\underline{\pi}_0\Omega^{d}_{\PP^1}\Sigma^{d}_{\PP^1}\mathcal{X}\to
\underline{\pi}_0\Omega^{d+1}_{\PP^1}\Sigma^{d+1}_{\PP^1}\mathcal{X}$
\end{center}
is an isomorphism for $d=2$.
\end{customthm}

\begin{customthm}{7}[Corollary \ref{EquivalenceHIMWandHIfr}]
The category of homotopy sheaves with generalized transfers, the category of MW-homotopy sheaves and the category of homotopy sheaves with framed transfers are equivalent:
\begin{center}
$\HIgtr(k)\simeq \HIMW(k) \simeq \HIfr(k)$.
\end{center}
\end{customthm}

\subsection*{Outline of the paper}

In Section \ref{SectionHomotopySheaves}, we follow \cite[Chapter 2]{Mor12} and recall the theory of unramified sheaves and how they are related to homotopy sheaves of abelian groups.
\par In Section \ref{SectionSheavesWithGeneralizedTransfers}, we define the notion of sheaves with generalized transfers and study the associated Rost-Schmid complex.
\par In Section \ref{SectionMorelConjecture}, we define the Bass-Tate transfer maps on a contracted homotopy sheaf $\M$ and prove the conjecture of Morel in the case of rational coefficients.
\par In Section \ref{SectionMWHomotopySheaves}, we recall the theory of sheaves with MW-transfers \cite{BCDFO} and prove that it is equivalent to the notion of sheaves with generalized transfers.
\par In Section \ref{SectionApplications}, we give some corollaries of Theorem \ref{EquivalenceHIMWandHIgtr}. In particular, we characterize the essential image of the functor ${\tilde{\gamma}_*:\HIMW(k)\to \HI(k)}$ that forgets MW-transfers and use the previous results to discuss the {\it conservativity conjecture} in $\AAA^1$-homotopy due to Bachmann and Yakerson (see \cite[Conjecture 1.1]{BachmannYakerson18} and \cite{Bach20}).

\subsection{Notation}\label{Notation}
Throughout the paper, we fix a (commutative) field $k$ and we assume moreover that $k$ is infinite perfect of characteristic not $2$. We need these assumptions in order to apply the cancellation theorem \cite[Chapter 4]{BCDFO} but we believe these restrictions could be lifted.
\par We denote by $\Grp$ and $\Ab$ the categories of (abelian) groups.
\par We consider only schemes that are essentially of finite type over $k$. All schemes and morphisms of schemes are defined over $k$. The category of smooth $k$-schemes of finite type is denoted by $\Sm$ and is endowed with the Nisnevich topology (thus, {\em sheaf} always means {\em sheaf for the Nisnevich topology}).
\par If $X$ is a scheme and $n$ a natural number, we denote by $X_{(n)}$ (resp. $X^{(n)}$) the set of point of dimension $n$ (resp. codimension $n$).
\par By a field $E$ over $k$, we mean {\em a $k$-finitely generated field $E$}. Since $k$ is perfect, notice that $\Spec E$ is essentially smooth over $S$. We denote by $\FFF_k$ the category of such fields.
\par Let $f:X\to Y$ be a morphism of schemes. Denote by $\LL_f$ (or $\LL_{X/Y}$) the virtual vector bundle over $Y$ representing the cotangent complex of $f$, and by $\LLL_f$ (or $\LLL_{X/Y}$) its determinant. Recall that if $p:X\to Y$ is a smooth morphism, then $\LL_p$ is (isomorphic to) $\TT_p=\Om_{X/Y}$ the space of (Kähler) differentials. If $i:Z\to X$ is a regular closed immersion, then $\LL_i$ is the normal cone $-\NN_ZX$. If $f$ is the composite $\xymatrix{ Y \ar[r]^i & \PP^n_X \ar[r]^p & X}$  with $p$ and $i$ as previously (in other words, if $f$ is lci projective), then $\LL_f$ is isomorphic to the virtual tangent bundle $i^*\TT_{\PP^n_X/X} - \NN_Y(\PP^n_X) $ (see also \cite[Section 9]{Fel18}).
\par Let $X$ be a scheme and $x\in X$ a point, we denote by $\LL_{x}=(\mathfrak{m}_x/\mathfrak{m}_x^2)^{\vee}$ and $\LLL_x$ its determinant. Similarly, let $v$ a discrete valuation on a field, we denote by $\LLL_{v}$ the line bundle $(\mathfrak{m}_v/\mathfrak{m}_v^2)^{\vee}$.
\par Let $E$ be a field (over $k$) and $v$ a valuation on $E$. We will always assume that $v$ is discrete. We denote by $\mathcal{O}_v$ its valuation ring, by $\mathfrak{m}_v$ its maximal ideal and by $\kappa(v)$ its residue class field. We consider only valuations of geometric type, that is we assume: $k\subset \mathcal{O}_v$, the residue field $\kappa(v)$ is finitely generated over $k$ and satisfies $\operatorname{tr.deg}_k(\kappa(v))+1=\operatorname{tr.deg}_k(E)$.

\par Let $E$ be a field. We denote by $\GW(E)$ the Grothendieck-Witt ring of symmetric bilinear forms on $E$. For any $a\in E^*$, we denote by $\ev{a}$ the class of the symmetric bilinear form on $E$ defined by $(X,Y)\mapsto aXY$ and, for any natural number $n$, we put $n_{\epsilon}=\sum_{i=1}^n \ev{-1}^{i-1}$.

\subsection*{Acknowledgment} I deeply thank my two PhD advisors Frédéric Déglise and Jean Fasel. Moreover, I would like to thank Tom Bachmann for precious comments on a draft.

\section{Homotopy sheaves} \label{SectionHomotopySheaves}

\subsection{Unramified sheaves} \label{SubsectionUnramifiedSheaves}
In this subsection, we summarize \cite[Chapter 2]{Mor12} and recall the basic results concerning unramified sheaves.

\begin{Def}
\begin{enumerate}
\item A sheaf of sets $\mathcal{S}$ on $\Sm$ is said to be $\AAA^1$-invariant if for any $X\in\Sm$, the map
\begin{center}
$\mathcal{S}(X)\to \mathcal{S}(\AAA^1_X)$
\end{center}
induced by the projection $\AAA^1\times X\to X$, is a bijection.
\item A sheaf of groups $\mathcal{G}$ on $\Sm$ is said to be {\em strongly $\AAA^1$-invariant} if, for any $X\in \Sm$, the map
\begin{center}
$H^i_{Nis}(X,\mathcal{G})\to H^i_{Nis}(\AAA^1\times X, \mathcal{G})$
\end{center} 
induced by the projection $\AAA^1\times X\to X$, is a bijection for $i\in \{0,1\}$.
\item A sheaf of abelian groups $M$ on $\Sm$ is said to be {\em strictly $\AAA^1$-invariant} if, for any $X\in \Sm$, the map
\begin{center}
$H^i_{Nis}(X,M)\to H^i_{Nis}(\AAA^1\times X,M)$
\end{center} 
induced by the projection $\AAA^1\times X\to X$, is a bijection for $i\in \NNN$.\end{enumerate}
\end{Def}

\begin{Rem} In the sequel, we work with $M$ a sheaf of groups. We could give more general definitions for sheaves of sets but, in practice, we need only the case of sheaves of abelian groups. In that case, we recall that a strongly $\AAA^1$-invariant sheaf of abelian groups is necessarily strictly $\AAA^1$-invariant (see \cite[Corollary 5.45]{Mor12}).
\end{Rem}

\begin{Def}
An unramified presheaf of groups $M$ on $\Sm$ is a presheaf of groups $M$ such that the following holds:
\begin{description}
\item[\namedlabel{itm:(0)}{(0)}] For any smooth scheme $X\in \Sm$ with irreducible components $X_\alpha$ ($\alpha\in X^{(0)}$), the canonical map $M(X)\to \prod_{\alpha \in X^{(0)}} M(X_{\alpha})$ is an isomorphism.
\item[\namedlabel{itm:(1)}{(1)}] For any smooth scheme $X\in \Sm$ and any open subscheme $U\subset X$ everywhere dense in $X$, the restriction map $M(X)\to M(U)$ is injective.
\item[\namedlabel{itm:(2)}{(2)}] For any smooth scheme $X\in \Sm$, irreducible with function field $F$, the injective map $M(X)\to \bigcap_{x\in X^{(1)}}M(\mathcal{O}_{X,x})$ is an isomorphism (the intersection being computed in $M(F)$).
\end{description}
\end{Def}

\begin{Exe}
Homotopy modules with transfers \cite{Deg10} and Rost cycle modules \cite{Rost96} define unramified sheaves. In characteristic not 2, the sheaf associated to the presheaf of Witt groups $X\to W(X)$ is unramified. 
\end{Exe}

We may give an explicit description of unramified sheaves on $\Sm$ in terms of their sections on fields $F\in \mathcal{F}_k$ and some extra structure. We will say that a function $M:\mathcal{F}_k\to \Grp$ is {\em continuous} if $M(F)$ is the filtering colimit of the groups $M(F_\alpha)$ where $F_\alpha$ run over the subfields of $F$ of finite type over $k$.
\begin{Def}[\cite{Mor12},Definition 2.6]
An unramified $\mathcal{F}_k$-datum consists of:
\begin{description}
\item[\namedlabel{itm:uD1}{uD1}]A continuous functor $M:\FFF_k \to \Grp$.
\item[\namedlabel{itm:uD2}{uD2}] For any field $F\in \FFF_k$ and any discrete valuation $v$ on $F$, a subgroup
\begin{center}
$M(\mathcal{O}_v)\subset M(F)$.
\end{center}
\item[\namedlabel{itm:uD3}{uD3}] For any field $F\in \FFF_k$ and any valuation $v$ on $F$, a map $s_v:M(\OO_v)\to M(\kappa(v))$, called the specialization map associated to $v$.

\end{description}
The previous data should satisfy the following axioms:
\begin{description}
\item[\namedlabel{itm:uA1}{uA1}] If $\iota: E\subset F$ is a separable extension in $\FFF_k$ and $w$ is a valuation on $F$ which restrict to a discrete valuation $v$ on $E$ with ramification index $1$, then the arrow $M(\iota)$ maps $M(\OO_v)$ into $M(\OO_w)$. Moreover, if the induced extension $\bar{\iota}:\kappa(v)\to \kappa(w)$ is an isomorphism, then the following square
\begin{center}
$\xymatrix{
M(\OO_v)
\ar[r]
\ar[d]
& 
M(\OO_w)
\ar[d]
\\
M(E)
\ar[r]
&
M(F)
}$
\end{center}
is cartesian.
\item[\namedlabel{itm:uA2}{uA2}] Let $X\in \Sm$ be an irreducible smooth scheme with function field $F$. If $x\in M(F)$, then $x$ lies in all but a finite number of $M(\OO_x)$ where $x$ runs over the set $X^{(1)}$ of points of codimension $1$.
\item[\namedlabel{itm:uA3(i)}{uA3(i)}] If $\iota:E\subset F$ is an extension in $\FFF_k$ and $w$ is a discrete valuation on $F$ which restricts to a discrete valuation $v$ on $F$, then $M(\iota)$ maps $M(\OO_v)$ into $M(\OO_w)$ and the following diagram
\begin{center}
$\xymatrix{
M(\OO_v)
\ar[r]
\ar[d]
&
M(\OO_w)
\ar[d]
\\
M(\kappa(v))
\ar[r]
&
M(\kappa(w))
}$
\end{center}
is commutative.

\item[\namedlabel{itm:uA3(ii)}{uA3(ii)}] If $\iota:E\subset F$ is an extension in $\FFF_k$ and $w$ a discrete valuation on $F$ which restricts to zero on $E$, then the map $M(\iota):M(E)\to M(F)$ has its image contained in $M(\OO_v)$. Moreover, if $\bar{\iota}:E\subset \kappa(w)$ denotes the induced extension, the composition 
$\xymatrix{
M(E)
\ar[r]
&
 M(\OO_v)
\ar[r]^{s_v}
&
M(\kappa(w))
}$ is equal to $M(\bar{\iota})$.
\item[\namedlabel{itm:uA4(i)}{uA4(i)}] For any smooth scheme $X\in \Sm$ local of dimension $2$ with closed point $z\in X^{(2)}$, and for any point $y_0\in X^{(1)}$ with $\bar{y}_0\in \Sm$, then $s_{y_0}:M(\OO_{y_0})\to M(\kappa(y_0)$ maps $\cap_{y\in X^{(1)}} M(\OO_{\bar{y}_0,z})$ into $M(\kappa(y_0))$.

\item[\namedlabel{itm:uA4(ii)}{uA4(ii)}] The composition
\begin{center}
$\bigcap_{y\in X^{(1)}} M(\OO_y)\to M(\OO_{\bar{y}_0,z})\to M(\kappa(z))$
\end{center} 
does not depend on the choice of $y_0$ such that $\bar{y}_0\in \Sm$.
\end{description}

\begin{Par}
An unramified sheaf $M$ defines in an obvious way an unramified $\mathcal{F}_k$-datum. Indeed, taking the evaluation on the field extensions of $k$ yields a restriction functor:
\begin{center}
$M:\mathcal{F}_k \to \Grp, F\mapsto M(F)$
\end{center}
such that, for any field $F$ with valuation $v$, we have an $M(\OO_v)\subset M(F)$ and a specialization map $s_v:M(\OO_v)\to M(\kappa(v))$ (obtained by choosing smooth models over $k$ for the closed immersion $\Spec \kappa(v) \to \Spec \OO_v$). We claim that this satisfies axioms \ref{itm:uA1},...,\ref{itm:uA4(ii)}.
\par Reciprocally, given an unramified $\FFF_k$-datum $M$ and $X\in \Sm$ an irreducible smooth scheme with function field $F$, we define the subset $M(X)\subset M(F)$ as the intersection $\bigcap_{x\in X^{(1)}} M(\OO_x)\subset M(F)$. We extend the definition for any $X$ so that property \ref{itm:(0)} is satisfied. Using the fact that any map $f:Y\to X$ between smooth schemes is the composition
$\xymatrix{
Y
\ar@{^{(}->}[r]
&
Y\times_k X
\ar@{->>}[r]
&
X}$
of closed immersion followed by a smooth projection, one can define an unramified sheaf $M:\Sm \to \Grp$. In short, we have the following theorem.
\end{Par}

\begin{The}\cite[Theorem 2.11]{Mor12}
The two functors described above define an equivalence between the category of unramified sheaves on $\Sm$ and that of unramified $\FFF_k$-data.
\end{The}

\end{Def}

\begin{Par}
From now on, we will not distinguish between the notion of unramified sheaves on $\Sm$ and that of unramified $\FFF_k$-datum. In the remaining subsection, we fix $M$ an unramified sheaf of groups on $\Sm$ and explain how it is related to strongly $\AAA^1$-invariant sheaves.
\end{Par}
\begin{Par}{\sc Notation} If $\phi:E\to F$ is an extension of fields, the map
\begin{center}
$M(\phi):M(E)\to M(F)$
\end{center}
is also denoted by $\res_{\phi}$, $\res_{F/E}$ or $\phi_*$.
\end{Par}
\begin{Par}
Let $F\in \FFF_k$ be a field and $v$ a (discrete) valuation on $F$. We define the pointed set
\begin{center}
$H^1_v(\OO_v;M)=M(F)/M(\OO_v)$.
\end{center}
This is a left $M(F)$-set. Moreover, for any point $y$ of codimension $1$ in $X\in \Sm$, we set $H^1_v(X,;M)=H^1_y(\OO_{X,y};M)$. By axiom \ref{itm:uA2}, if $X$ is irreducible with function field $F$, the induced left action of $M(F)$ on $\prod_{y\in X^{(1)}}H^1_y(X,M)$ preserves the weak-product
\begin{center}
$\prod_{y\in X^{(1)}}'H^1_y(X,M) \subset \prod_{y\in X^{(1)}}H^1_y(X,M)$
\end{center}
where the weak-product $\prod_{y\in X^{(1)}}'H^1_y(X,M)$ means the set of families for which all but a finite number of terms are the base point of $H^1_y(X,M)$. By definition and axiom \ref{itm:uA2}, the isotropy subgroup of this action of $M(F)$ on the base point of $\prod_{y\in X^{(1)}}'H^1_y(X,M)$ is exactly $M(X)=\cap_{y\in X^{(1)}}M(\OO_{X,y})$. We summarize this property by saying that the diagram
\begin{center}
$
1
\to
M(X)
\to
M(F)
\Rightarrow
\prod_{y\in X^{(1)}}'H^1_y(X,M)$
\end{center}
is exact.
\end{Par}
\begin{Def} For any point $z$ of codimension $2$ in a smooth scheme $X$, we denote by $H_z^2(X,M)$ the orbit set of $\prod_{y\in X_z^{(1)}}'H^1_y(X,M)$ under the left action of $M(F)$ where $F\in \FFF_k$ is the function field of $X_z$.
\end{Def}

\begin{Par}
For an irreducible essentially smooth scheme $X$ with function field $F$, we define the boundary $M(F)$-equivariant map
\begin{center}
$\prod_{y\in X^{(1)}}'H^1_y(X,M) \to \prod_{z\in X^{(2)}}H^2_z(X,M)$
\end{center}
by collecting together the compositions
\begin{center}
$\prod_{y\in X^{(1)}}'H^1_y(X,M)
\to 
\prod_{y\in X_z^{(1)}}'H^1_y(X,M)
\to
H^2_z(X,M)$
\end{center}
for each $z\in X^{(2)}$.
\par It is not clear in general whether or not the image of the boundary map is always contained in the weak product $\prod_{z\in X^{(2)}}'H^2_z(X,M)$. For this reason we introduce the following axiom:
\begin{description}
\item[\namedlabel{itm:uA2'}{uA2'}] For any irreducible essentially smooth scheme $X$, the image of the boundary map
\begin{center}
$\prod_{y\in X^{(1)}}'H^1_y(X,M) \to \prod_{z\in X^{(1)}}H^2_z(X,M)$
\end{center}
is contained in the weak product $\prod_{z\in X^{(2)}}'H^2_z(X,M)$.
\end{description}
\end{Par}
\begin{Par}
From now on we assume that $M$ satisfies \ref{itm:uA2'}. For any smooth scheme $X$ irreducible with function field $F$, we have a complex $C^*(X,M)$
\begin{center}
$1\to M(X) \to M^{(0)}(X) \Rightarrow M^{(1)}(X) \to M^{(2)}(X)$
\end{center}
where $M^{(0)}(X)=\prod_{x\in X^{(0)}}'M(\kappa(x))=\prod_{x\in X^{(0)}}M(\kappa(x))$, $M^{(1)}(X)=\prod_{y\in X^{(1)}}'H^1_y(X,M)$ and $M^{(2)}(X)=\prod_{z\in X^{(2)}}'H^2_z(X,M)$. By construction, this complex is exact (in an obvious sense, see \cite[Definition 2.20]{Mor12}) for any (essentially) smooth local scheme of dimension $\leq 2$.
\end{Par}
\begin{Def}
A strongly unramified $\FFF_k$-data is an unramified $\FFF_k$-data satisfying \ref{itm:uA2'} and the following axioms on $M$:
\begin{description}
\item[\namedlabel{itm:uA5(i)}{uA5(i)}] For any separable finite extension $\iota:E\subset F$ in $\FFF_k$, any discrete valuation $w$ on $F$ which restricts to a discrete valuation $v$ on $E$ with ramification index $1$, and such that the induced extension $\bar{\iota}:\kappa(v)\to \kappa(w)$ is an isomorphism, the commutative square of groups
\begin{center}
$\xymatrix{
M(\OO_v)
\ar[r]
\ar[d]
&
M(E)
\ar[d]
\\
M(\OO_w)
\ar[r]
&
M(F)
}$
\end{center}
induces a bijection $H^1_w(\OO_w,M)\simeq H^1_v(\OO_v,M)$.
\item[\namedlabel{itm:uA5(ii)}{uA5(ii)}]
For any étale morphism $X'\to X$ between smooth local $k$-schemes of dimension $2$, with closed point respectively $z'$ and $z$, inducing an isomorphism on the residue fields $\kappa(z)\simeq \kappa(z')$, the pointed map
\begin{center}
$H^2_z(X,M)\to H^2_{z'}(X',M)$
\end{center}
has trivial kernel.
\item[\namedlabel{itm:uA6}{uA6}] For any localization $U$ of a smooth $k$-scheme at some point $u$ of codimension $\leq 1$, the complex:
\begin{center}
$1\to M(\AAA^1_U) \to M^{(0)}(\AAA^1_U) \Rightarrow M^{(1)}(\AAA^1_U) \to M^{(2)}(\AAA^1_U)$
\end{center}
is exact. Moreover, the morphism $M(U)\to M(\AAA^1_U)$ is an isomorphism.
\end{description}
\end{Def}

\begin{The}\cite[Theorem 2.27]{Mor12} \label{EquivalenceHomotopySheaves}
There is en equivalence between the category of strongly $\AAA^1$-invariant sheaves of groups on $\Sm$ and that of strongly unramified $\FFF_k$-data of groups on $\Sm$.
\end{The}

\begin{Def} A strongly $\AAA^1$-invariant Nisnevich sheaf of abelian groups is called a {\em homotopy sheaf}. We denote by $\HI(k)$ the category of homotopy sheaves and natural transformations of sheaves.
\end{Def}

\begin{Par}{\sc Monoidal structure} \label{HImonoidal}
Recall that there is a canonical adjunction of categories
\begin{center}

$\xymatrix{
\DAeff(k)
\ar@<1ex>[r]^-{\OO}
&
D(\operatorname{Sh}(\Sm))
\ar@<1ex>[l]^-{\pi_{\AAA^1}}
}$

\end{center}
where $\DAeff(k)$ the effective $\AAA^1$-derived category and $D(\operatorname{Sh}(\Sm))$ is the derived category of complexes of sheaves over $\Sm$ (see \cite[§5]{CD12}). Thanks to Morel's $\AAA^1$-localization theorem, we can prove that there is a unique t-structure on $\DAeff(k)$ such that the forgetful functor $\OO$ is t-exact and that the category of homotopy sheaves $\HI(k)$ is equivalent to the heart $(\DAeff(k))^{\heartsuit}$ for this t-structure (in particular, $\HI(k)$ is a Grothendieck category). Since the canonical tensor product $\otimes_{\DAeff}$ is right t-exact, it induces a monoidal structure on $\HI(k)$. Precisely, if $F,G\in \HI(k)$ are two homotopy sheaves, then their tensor product is
\begin{center}
$F\otimes_{HI} G=H_0^{\AAA^1}\circ \OO \circ \pi_{\AAA^1}(F\otimes_{\DAeff} G)$
\end{center}
where $H_0^{\AAA^1}$ is the homology object in degree 0 for the homotopy t-structure.
\end{Par}

\subsection{Contracted homotopy sheaves} \label{SubsectionContractedHomotopySheaves}

\begin{Par}
In this section, we fix $M\in \HI(k)$ a homotopy sheaf. Recall that the contraction $M_{-1}$ is by definition the sheaf of abelian groups
\begin{center}
$X\mapsto \ker(M(\Gm \times X) \to M(X))$.
\end{center}
Equivalently, we have
\begin{center}
$\M=\Hominter(\Gm,M)$
\end{center}
where $\Hominter$ is the internal hom-object of $\HI(k)$. According to \cite[Lemma 2.32]{Mor12}, the sheaf $M_{-1}$ is also a homotopy sheaf and is called a {\em contracted homotopy sheaf}.
\end{Par}

\begin{Par} \label{GWstructureContracted}
Any unit map $a:\Gm\to \Gm$ induces a morphism 
\begin{center}
$\ev{a}:\Hominter(\Gm,M)\to \Hominter(\Gm,M)$
\end{center}
and thus a $\GW$-module structure on $\M$ (see \cite[Lemma 3.49]{Mor12}). Moreover, we have a bilinear pairing
\begin{center}
$\kMW_1 \times M_{-1} \to M$ \\
$([u],\mu)\mapsto [u]\cdot \mu$. 
\end{center}
\end{Par}

\begin{Par} \label{DefMtwisted}
Let $N\in \HI(k)$ be another homotopy sheaf and assume $N$ is equipped with a $\GW$-module structure. Let $X\in \Sm$ be a smooth scheme and $\LL$ a vector line bundle on $X$. As usual, we define the twist of $N$ by $\LL$:
\begin{center}
$N\{\LL\}=N\otimes_{\ZZ[\Gm]}\ZZ[\LL^\times]$
\end{center}
as the sheaf associated to the presheaf on the Zariski site $X_{\txt{Zar}}$:
\begin{center}
$U\mapsto N(U)\otimes_{\ZZ[\OO_X(U)^\times]}\ZZ[\LL_U^\times]$
\end{center}
where $\LL_U^\times$ is the set of isomorphisms between $\OO_U$ and $\LL_U$ (which may be empty). We put $N(X,\LL)=\Gamma(X,N\{\LL\})$ and remark that, if $X$ is affine, then $N(X,\LL)=N\otimes_{\ZZ[\Gm]}\ZZ[\LL^\times]$.
In particular, this definition applies to the sheaf $N=\M$.
\end{Par}

\begin{Par}{\sc Cohomology with support exact sequence}  \label{HILongExactSequence} Let $M\in \HI(k)$ be a homotopy sheaf. For any closed immersion $i:Z\to X$ of smooth schemes over $k$, with complementary open immersion $j:U\to X$, there exists a canonical cohomology with support exact sequence of the form:
\begin{center}
$\xymatrix{
\Gamma_Z(X,M)
\ar[r]^-{i_*}
&
M(X)
\ar[r]^-{j^*}
&
M(U)
\ar[r]^-{\partial}
&
H^1_Z(X,M)
\ar[r]
&
\dots
}$
\end{center}
\end{Par}

\begin{Par} \label{ResidueMapsMoinsUn} The purity isomorphism (more precisely: the axiom \ref{itm:uA5(i)} and \cite[Lemma 3.50]{Mor12}) implies that for any discrete valuation $v$ on a field $F\in \FFF_k$, one has a canonical bijection
\begin{center}
$H^1_v(\OO_v,M)\simeq \M(\kappa(v),\LLL_{v})$
\end{center}
and thus obtain a residue map
\begin{center}
$\partial_v:M(F)\to \M(\kappa(v),\LLL_{v})$.
\end{center}

\end{Par}

\begin{Pro} \label{LocalizationContractedHI}

Let $M\in \HI(k)$ be a homotopy sheaf and consider the following commutative square
\begin{center}

$\xymatrix{
T \ar@{^{(}->}[r]^k  \ar@{^{(}->}[d]_q  & Y \ar@{^{(}->}[d]^p \\
Z \ar@{^{(}->}[r]_i & X 
}$

\end{center}
of closed immersions of separated schemes over $k$. We have the following diagram
\begin{center}

$\xymatrixcolsep{1pc} \xymatrix{
\Gamma_T(X,M) \ar[r]^{k_*} \ar[d]_{q_*} &
\Gamma_Y(X,M) \ar[r]^-{k'^*} \ar[r] \ar[d]^{p_*} &
\Gamma_{Y-T}(X-Z,M) \ar[r]^{\partial_k} \ar[d]^{\tilde{p}_*} &
H^1_T(X,M) \ar[d]^{q_*} \\
\Gamma_Z(X,M) \ar[r]^{i_*} \ar[d]^{q'^*} & 
M(X) \ar[r]^{i'^*} \ar[d]^{p'_*} &
M(X- Z) \ar[d]^{\tilde{p}'^*} \ar[r]^{\partial_i} &
H^1_Z(X,M) \ar[d]^{q'^*} \\
\Gamma_{Z-T}(X-Y,M) \ar[r]^-{\tilde{i}_*} \ar[d]^{\partial_k} &
M(X- Y) \ar[r]^-{\tilde{i}'^*} \ar[d]^{\partial_{{p}}} &
M(X- (Y\cup Z)) \ar@{}[rd]|-{(*)} \ar[d]^{\partial_{\tilde{p}}} \ar[r]^-{\partial_{\tilde{i}}} &
H^1_{Z-T}(X- T,M) \ar[d]^{\partial_{\tilde{k}}} \\
H^1_T(X,M) \ar[r]_{k_*} & 
H^1_Y(X,M) \ar[r]_-{k'^*} &
H^1_{Y-T}(X- Z,M) \ar[r]_-{\partial_k} &
H^{2}_T(X,M)
}$
\end{center}
with obvious maps. Each squares of this diagram is commutative except for $(*)$ which is anti-commutative.
\end{Pro}
\begin{proof}
This is a classical consequence of the octahedron axiom.
\end{proof}

The next proposition is a direct application of the functoriality in $M$ of the cohomology with support exact sequence.
\begin{Pro} \label{eR3eHI}
Let $i:Z\to X$ a closed immersion with complementary open immersion ${j:U\to X}$. Let $\alpha \in \kMW_1(X)$. Then the following diagram is commutative:
\par 

\begin{center}

$\xymatrix{
\Gamma_Z(X,\M) \ar[r]^-{i_*} \ar[d]^{  \gamma_{\alpha}} &
\M(X) \ar[r]^{j^*} \ar[d]^{\gamma_{\alpha}} &
M_{-1}(U) \ar[r]^{\partial_{Z,X}} \ar[d]^{\gamma_{\alpha}} &
H^1_Z(X,\M) \ar[d]^{\gamma_{\alpha}} \\
\Gamma_Z(X,M)  \ar[r]^{i_*}&
M(X)  \ar[r]^{j^*}&
M(U)  \ar[r]^-{\partial_{Z,X}} &
H^1_Z(X,M)
}$

\end{center}

where $\gamma_{\alpha}$ is the multiplication map (see \ref{GWstructureContracted}).
\end{Pro}

\section{Sheaves with generalized transfers} \label{SectionSheavesWithGeneralizedTransfers}

\subsection{Morel's axioms} \label{SubsectionSheavesWithGeneralizedTransfers}

The following definition is a slightly improved version of Morel's definition in \cite[Definition 5.7]{Mor11}. It is directly inspired by Rost's theory of cycles modules. Lastly, it can be seen as an effective counter-part of our own axiomatic of MW-cycle modules in \cite{Fel18}.

\begin{Def} \label{DefGeneralizedTransfers}
Let $M$ be a homotopy sheaf (i.e. a strongly $\AAA^1$-invariant Nisnevich sheaf of abelian groups on $\Sm$). We say that $M$ has a structure of generalized transfers if $M$ has a structure of $\GW$-modules and satisfies the following datum
\begin{description}
\item[\namedlabel{itm:eD2}{eD2}] For each finite extension $\phi:E\to F$ in $\mathcal{F}_k$ a map
\begin{center}
$\Tr_{\phi}=\Tr_{F/E}:M(F,\LLL_{F/k})\to M(E,\LLL_{E/k})$
\end{center} 
called the transfer morphism from $F$ to $E$.
\end{description}
In addition, this datum satisfies the following axioms:
\begin{description}

\item[\namedlabel{itm:eR1b}{eR1b}] $\Tr_{\Id_E}=\Id_{M(E)}$ and for any composable finite morphisms $\phi$ and $\psi$ in $\mathcal{F}_k$, we have 
\begin{center}
$\Tr_{\psi\circ \phi}=\Tr_{\phi}\circ \Tr_{\psi}$.

\end{center}

\item [\namedlabel{itm:eR1c}{eR1c}] Consider $\phi:E\to F$ and $\psi:E\to L$ with $\phi$ finite. Let $R$ be the ring $F\otimes_E L$. For each $\mathfrak{p}\in \Spec R$, let $\phi_\mathfrak{p}:L\to R/\mathfrak{p}$ and $\psi_\mathfrak{p}:F\to R/\mathfrak{p}$ be the morphisms induced by $\phi$ and $\psi$. One has
\begin{center}
$M(\psi)\circ \Tr_{\phi}=\displaystyle \sum_{\mathfrak{p}\in \Spec R}e_{\mathfrak{p},\epsilon} \Tr_{\phi_\mathfrak{p}}\circ M(\psi_\mathfrak{p})$
\end{center}
where $e_{p,\epsilon}=\sum_{i=1}^{e_p}\ev{-1}^{i-1}$ is the quadratic form associated to the length $e_p$ of the localized ring $R_{(\mathfrak{p})}$.

\item [\namedlabel{itm:eR2}{eR2}] Let $\psi:E\to F$ be a finite extension of fields.

\item [\namedlabel{itm:eR2b}{eR2b}] For $\ev{a}\in \GW(E)$ and $\mu \in M(F,\LLL_{F/k})$, one has 	$\Tr_{F/E}\big(\ev{\psi(a)}\cdot \mu\big)=\ev{a}\cdot \Tr_{F/E}(\mu)$.

\item [\namedlabel{itm:eR2c}{eR2c}]  For $\ev{a}\in \GW(F,\LLL_{F/k})$ and $\mu \in M(E)$, one has 	$\Tr_{F/E}\big(\ev{a}\cdot \res_{F/E}(\mu)\big)=\Tr_{F/E}(\ev{a})\cdot \mu$.

\item [\namedlabel{itm:eR3b}{eR3b}] Let $\phi:E\to F$ be a finite extension of fields and let $v$ be a valuation on $E$. For each extension $w$ of $v$, we denote by $\phi_w:\kappa(v)\to \kappa(w)$ the induced morphism. We have
\begin{center}
$\partial_v\circ \Tr_{\phi}=\sum_w \Tr_{\phi_w}\circ \partial_w$
\end{center}
where $\partial_v:M(E,\LLL_{E/k})\to M(\kappa(v),\LLL_{\kappa(v)/k})$ and $\partial_w:M(F,\LLL_{F/k})\to M(\kappa(w),\LLL_{\kappa(w)/k})$ are the residue maps defined in \ref{ResidueMapsMoinsUn}.

\end{description}

\end{Def}

\begin{Rem}
Our definition differs from Morel's in two ways. First, we have taken into account the twists naturally arising (this is not really important if one works in zero characteristic). Second, the axiom \textbf{A3} of \textit{loc. cit.} is replaced by \ref{itm:eR3b} (these two axioms are equivalent in some cases).
\end{Rem}

\begin{Rem}
We know from \ref{EquivalenceHomotopySheaves} that homotopy sheaves can be understood as certain functors on function fields. Taking into account this fact, our axioms are indeed {\it effective} variants of that of Milnor-Witt cycle modules. In fact, we will see that they correspond to homotopy sheaves with Milnor-Witt transfers, as MW-cycle modules correspond to homotopy modules. (This explains our choice of numbering of the axioms.)
\end{Rem}

\begin{Rem}
A homotopy sheaf with generalized transfers is a particular case of a sheaf with $\AAA^1$-transfers as defined in \cite[§5]{BachmannYakerson18}.
\end{Rem}

\begin{Exe} \label{GeneralizedTransfersOnContraction}
\begin{enumerate}
\item The Milnor-Witt sheaf $\KMW$ has a structure of generalized transfers (\cite[§4.2]{Mor12}, or \cite[Theorem 4.13]{Fel18}, see also Theorem \ref{R1c_fort}). 
\item  If $M$ is a homotopy sheaf with generalized transfers, then so is $\M$.
\item In the next section, we will give conditions for a contracted homotopy sheaf $\M$ to be equipped with a structure of generalized transfers.
\end{enumerate}
\end{Exe}

\begin{Def} A map between homotopy sheaves with Morel generalized transfers is a natural transformation commuting with the transfers.
We denote by $\HIgtr(k)$ the category of homotopy sheaves with (Morel) generalized transfers over $k$.
\end{Def}

\begin{Par}{\sc Rost-Schmid complex}     \label{e2.0.1}
Let $M\in \HIgtr(k)$ and let $X$ be a scheme essentially of finite type over $k$. We define the Rost-Schmid complex as the graded abelian group defined for any $n\in \NNN$ by:
\begin{center}
$C^n(X,M)=\bigoplus_{x\in X^{(n)}}M_{-n}(\kappa(x),\LLL_{\kappa(x)/k})$.

\end{center}
The transition maps $d$ are defined as follows. If $X$ is normal with generic point $\xi$, then for any $x\in X^{(1)}$ the local ring of $X$ at $x$ is a valuation ring so that we have a map $\partial_x: M_{-n}(\kappa(\xi), \LLL_{\kappa(\xi)/k}) \to M_{-n-1}(\kappa(x), \LLL_{\kappa(x)/k})$ for any $n$.

Now suppose $X$ is a scheme essentially of finite type over $k$ and let $x,y$ be two points in $X$. We define a map
\begin{center}
$\partial^x_y:M_{-n}(\kappa(x),\LLL_{\kappa(x)/k}) \to M_{-n-1}(\kappa(y),\LLL_{\kappa(y)/k})$
\end{center}
as follows. Let $Z=\overline{ \{x\}}$. If $y\not \in Z$, then put $\partial^x_y=0$. If $y\in Z$, let $\tilde{Z}\to Z$ be the normalization and put
\begin{center}
$\partial^x_y=\displaystyle \sum_{z|y} \Tr_{\kappa(z)/\kappa(y)}\circ \, \partial_z$
\end{center}
with $z$ running through the finitely many points of $\tilde{Z}$ lying over $y$.
\par Thus, we may define a differential map
\begin{center}
$d=\sum_{x,y}\partial_y^x:C^n(X,M)\to C^{n+1}(X,M)$.
\end{center}

\end{Par}
\begin{Pro}
For any homotopy sheaf $M$ and any scheme essentially of finite type $X$, the Rost-Schmid complex  $C^*(X,M)$ is a complex.
\end{Pro}
\begin{proof}
See \cite[Theorem 5.31]{Mor12}.
\end{proof}

\begin{Def}
For any homotopy sheaf $M$ and any scheme essentially of finite type $X$, the cohomology groups associated to the Rost-Schmid complex are denoted by $A^*(X,M)$.

\end{Def}
We now define the usual basic maps for the Rost-Schmid complex and prove that they are complex morphisms in some special cases.
\begin{Par}
{\sc Pullback} \label{RostSchmidPullback}
Let $M\in \HIgtr(k)$. Let $f:X\to Y$ be an {essentially smooth} morphism of schemes essentially of finite type. Suppose $X$ connected and denote by $s$ the relative dimension of $f$. Define
\begin{center}
$f^*:C^*(Y,M) \to C^{*}(X,M\otimes \LLL_{X/Y}^{\vee})$
\end{center}
as follows. If $f(x)=y$, then $(f^*)^y_x=\Theta\circ \res_{\kappa(x)/\kappa(y)}$, where $\Theta$ is the canonical isomorphism induced by $\LLL_{\Spec\kappa(x)/\Spec\kappa(y)}\simeq \LLL_{X/Y} \times_X \Spec \kappa(x)$. Otherwise, $(f^*)^y_x=0$. If $X$ is not connected, take the sum over each connected component.

\end{Par}

\begin{Par}{\sc Pushforward} \label{RostSchmidPushforward}Let $M\in \HIgtr(k)$. Let $f:X\to Y$ be a morphism between schemes essentially of finite type over $k$ and assume that $X$ is connected (if $X$ is not connected, take the sum over each connected component). Let $d=\dim(Y)-\dim(X)$. We define
\begin{center}
$f_*:C^*(X,M)\to C^{*-d}(Y,M_{-d})$
\end{center}
as follows. If $y=f(x)$ and $\kappa(x)$ is finite over $\kappa(y)$, then put $(f_*)^y_x=\Tr_{\kappa(x)/\kappa(y)}$ where $\Tr_{\kappa(x)/\kappa(y)}$ is the transfer map \ref{itm:eD2} if $n=0$, and is the transfer map coming from the induced generalized transfers on $\M$ if $n>0$. Otherwise, put $(f_*)^y_x=0$.
\end{Par}

\begin{Par}{\sc $\GW$-action} \label{RostSchmidGWaction}
Let $M\in \HIgtr(k)$. Let $X$ be a scheme essentially of finite type and $a\in \OO_X^*$ a global unit. Define a morphism
\begin{center}
$\ev{a}:C^*(X,M)\to C^*(X,M)$
\end{center}
as follows. Let $x,y\in X^{(p)}$ and $\rho\in M_{-*}(\kappa(x),\LLL_{x/k})$. If $x=y$, then $\ev{a}^y_x(\rho)=\ev{a(x)}\cdot \rho$. Otherwise, $\ev{a}^y_x(\rho)=0$.
\end{Par}

\begin{Par}{\sc Boundary maps} \label{BoundaryMaps}
Let $M\in \HIgtr(k)$. Let $X$ be a scheme essentially of finite type over $k$, let $i:Z\to X$ be a closed immersion and let $j:U=X\setminus Z \to X$ be the inclusion of the open complement. We will refer to $(Z,i,X,j,U)$ as a boundary triple and define
\begin{center}

$\partial=\partial^U_Z:C_p(U,M) \to C_{p-1}(Z,M)$
\end{center}
by taking $\partial^x_y$ to be as the definition in \ref{e2.0.1} with respect to $X$. The map $\partial^U_Z$ is called the boundary map associated to the boundary triple, or just the boundary map for the closed immersion $i:Z\to X$.

\end{Par}

\par We now fix $M\in \HIgtr(k)$ a homotopy sheaf with generalized transfers and study the morphisms defined on the Rost-Schmid complex $C^*(?,M)$.

\begin{Pro}[Functoriality and base change] \label{Prop4.1eff}

\begin{enumerate}
\item Let $f:X\to Y$ and $f':Y\to Z$ be two morphisms of schemes essentially of finite type. Then
\begin{center}
$(f'\circ f)_*=f'_* \circ f_*$.

\end{center} 
\item Let $g:Y\to X$ and $g':Z\to Y$ be two essentially smooth morphisms. Then (up to the canonical isomorphism given by $\LLL_{Z/X}\simeq \LLL_{Z/Y}+(g')^*\LLL_{Y/Z}$):
\begin{center}

$(g\circ g')^*=g'^*\circ g^*$.
\end{center}
\item Consider a pullback square
\begin{center}

$\xymatrix{
U \ar[r]^{g'} \ar[d]_{f'} & Z \ar[d]^f \\
Y \ar[r]_{g} & X
}$
\end{center}
with $f,f',g,g'$ as previously. Then
\begin{center}

$g^*\circ f_* =  f'_* \circ g'^*$
\end{center}
up to the canonical isomorphism induced by $\LLL_{U/Z} \simeq \LLL_{Y/X}\times_Y U$.
\end{enumerate}
\end{Pro}
\begin{proof}
This follows as in \cite[Proposition 6.1]{Fel18} from \ref{itm:eR1b} and \ref{itm:eR1c}.
\end{proof}

\begin{Pro}\label{Prop4.6eff}
\begin{enumerate}
\item[(i)] Let $f:X\to Y$ be a finite morphism of schemes essentially of finite type. Then
\begin{center}

$d_Y\circ f_*= f_*\circ d_X$.
\end{center}
\item[(ii)] Let $g:Y\to X$ be an essentially smooth morphism. Then
\begin{center}

$g^*\circ d_X=d_Y \circ g^*$.
\end{center}
\item[(iii)] Let $a$ be a unit on $X$. Then
\begin{center}

$d_X \circ \ev{a}=\ev{a}\circ d_X$
\end{center}
\item[(iv)] Let $(Z,i,X,j,U)$ be a boundary triple. Then
\begin{center}

$d_Z\circ \partial^U_Z=-\partial^U_Z\circ d_U$.
\end{center}

\end{enumerate}

\end{Pro}
\begin{proof}
As in \cite[Proposition 6.6]{Fel18}, $(ii),(iii)$ and $(iv)$ follow easily from the definitions. The assertion $(i)$ is nontrivial since our axioms are weaker than in the stable case studied in \cite[§6]{Fel18}. Fortunately, Morel proved that the map
\begin{center}
$g_*:C_*(Y,N_{-1})\to C_*(X,N_{-1})$
\end{center}
is a morphism of complexes when $N$ is a homotopy sheaf (see \cite[Corollary 5.30]{Mor12})\footnote{It seems that \textit{loc. cit.} depends on Conjecture \ref{ConjectureMoinsUn} which is true in our case thanks to \ref{itm:eR1b}.}. The proof can be adapted almost verbatim if we replace the contracted homotopy sheaf $N_{-1}$ by any sheaf with generalized transfers $M$ (the use of \cite[Theorem 5.19]{Mor12} is replaced by \ref{itm:eR3b}).
\end{proof}

\begin{Rem}

According to the previous proposition, the morphisms $f_*$ for $f$ finite, $g^*$ for $g$ essentially smooth, multiplication by $\ev{a}$ commute with the differentials. We use the same notations to denote the induced morphisms on the cohomology groups $A^*(X,M)$.

\end{Rem}

\begin{Par} {\sc Monoidal structure} 
\label{GeneralizedTransfersMonoidal} Let $M,M'\in \HIgtr(k)$ be two homotopy sheaves with respective generalized transfers $\Tr^{M'}$ and $\Tr^{M'}$. Consider the homotopy sheaf $M\otimes_\HI M'$ and define the transfer maps by
\begin{center}
$\Tr^{M\otimes_\HI M'}=\Tr^{M}\otimes_\HI \Tr^{M'}$
\end{center}
where $\otimes_{\HI}$ is the tensor product defined in \ref{HImonoidal}.
It is easy to check that this produces a monoidal structure on $\HIgtr(k)$ such that the forgetful functor
\begin{center}
$\HIgtr(k) \to \HI(k)$
\end{center}
is monoidal.

\end{Par}

\subsection{MW-transfers structure} \label{SubsectionMWTransfersStructure}

\begin{Def}
Let $M\in \HIgtr(k)$ be a homotopy sheaf with generalized transfers. For any smooth scheme $X$, denote by
\begin{center}
$\tilde{\Gamma}_*(M)(X)=A^0(X,M\otimes (\LLL_{X/k})^{\vee})$.
\end{center}
This  defines a presheaf $\tilde{\Gamma}_*(M)$.
\end{Def}

\begin{The} \label{HIgtrToHIMW}
Let $M\in \HIgtr(k)$ be a homotopy sheaf with generalized transfers. Then the contravariant functor $\tilde{\Gamma}_*(M):X\mapsto \tilde{\Gamma}_*(M)(X)$ induces a presheaf on $\Cortilde_k$.\footnote{See Subsection \ref{MWRecollection} for more details on MW-transfers.}
\end{The}
\begin{proof}

Let $X,Y$ be two smooth schemes (which may be assumed to be connected without loss of generality) and $T\subset X\times Y$ be an admissible subset (see Definition \ref{DefAdmissiblSet}). Let $\beta\in \tilde{\Gamma}_*(M)(Y)$ and $\alpha \in \CHt^{d_Y}_T(X\times Y,\omega_Y)$. We set
\begin{center}
$\alpha^*(\beta)=(p_{X})_*(\alpha\cdot p_Y^{*}(\beta))$
\end{center}
where $p_X:T\to X$ and $p_Y:X\times Y \to  Y$ are the canonical morphisms, and where the product $\alpha\cdot p_Y^{*}(\beta)$ is defined by linearity thanks to \ref{RostSchmidGWaction} and Proposition <\ref{Prop4.6eff}(iii). We remark that the map $(p_X)_*$is well-defined\footnote{We have used Voevodsky's trick here: $X\times Y\mapsto X$ is not finite, but its restriction $p_X:T\to X$ is finite by assumption.} thanks to Proposition \ref{Prop4.6eff}(i). This yields to an application $\alpha^*$ which is additive.
\par If $T\subset T'$, we have a commutative diagram 
\begin{center}
$\xymatrix{
\CHt^{d_Y}_T(X\times Y,\omega_Y) \ar[r] \ar[rd]_{p_{X,*}}&
\CHt^{d_Y}_{T'}(X\times Y,\omega_Y) \ar[d]^{p_{X*}}
\\
 &
\CHt^0(X)
}$
\end{center}
with obvious morphisms. Thus $\alpha \mapsto \alpha^*$ defines a map $\Cortilde_k(X,Y)\to \Hom_{\mathcal{A}b}(\tilde{\Gamma}_*(M)(Y),\tilde{\Gamma}_*(M)(X))$. It remains to check that this map preserves the respective compositions. Consider the diagram
\begin{center}
$\xymatrix{
X\times Z  \ar@/^1pc/[rrrd]^{r_Z} \ar@/_1pc/[dddr]_{r_X}&
      	&
	&
	\\
	&
X\times Y \times Z   
\ar[ul]^{q_{XZ}}
\ar[r]^{q_{YZ}}
\ar[d]_{q_{XY}}&
Y\times Z 
\ar[r]^{q_Z}
\ar[d]^{p_Y}&
Z \\
	&
X\times Y 
\ar[r]_{q_Y}
\ar[d]_{p_X}&
Y &
	\\
	&
X. &
	&
	}$
\end{center}
Let $\alpha_1 \in \CHt^{d_Y}_{T_1}(X\times Y,\omega_Y)$ and $\alpha_2 \in \CHt^{d_Z}_{T_2}(X\times Z,\omega_Y)$ be two correspondences, with $T_1\subset X\times Y$ and $T_2\subset Y\times Z$ admissible. Moreover, let $\beta \in M(Z)$. By definition, we have
\begin{center}
$(\alpha_2\circ \alpha_1)^*(\beta)=(r_{X})_*[(q_{XY})_*(p^*_{XY}\alpha_1\cdot q^*_{YZ}\alpha_2)\cdot r^*_Z\beta].$
\end{center} 
Using the projection formula \ref{itm:eR2b}, we have
\begin{center}
$
\begin{array}{lcl}
(r_{X})_*[(q_{XY})_*(p^*_{XY}\alpha_1\cdot q^*_{YZ}\alpha_2)\cdot r^*_Z\beta] & = & (r_{X})_*[(q_{XY})_*(p^*_{XY}\alpha_1 \cdot q^*_{YZ}\alpha_2 \cdot q^*_{XY}r^*_Z\beta)] 
 \\
& = & 
(p_X)_*(p_{XY})_*(p^*_{XY}\alpha_1 \cdot q^*_{YZ} \alpha_2 \cdot q^*_{XZ}r^*_Z\beta).
\end{array} 
$

\end{center}
On the other hand,
\begin{center}
$
\begin{array}{lcl}
\alpha_1^*\circ \alpha_2^*(\beta)
 & = & \alpha_1^*((p_Y)_*(\alpha_2\cdot q^*_Z\beta)) 
 \\
& = & (p_X)_*(\alpha_1 \cdot q^*_Y(p_Y)_*(\alpha_2 \cdot q^*_Z\beta)) 
\end{array} 
$
\end{center}
By base change \ref{Prop4.1eff}, $q^*_Y(p_Y)_*=(p_{XY})_*q^*_{YZ}$ and it follows (using the projection formula once again) that
\begin{center}
$
\begin{array}{lcl}
\alpha_1^* \circ \alpha^*_2(\beta)
 & = & 
(p_X)_* (\alpha_1\cdot (p_{XY})_* (q^*_{YZ}\alpha_2\cdot q^*_{YZ} q^*_Z \beta)) 
 \\
& = &  (p_X)_*(p_{XY})_*(p^*_{XY}\alpha_1 \cdot q^*_{YZ} \alpha_2 \cdot q^*_{XZ}r^*_Z\beta).
\end{array} 
$
\end{center}
Hence the result.

\end{proof}

\par We have proved the following theorem.
 \begin{The} \label{MWModuleFM}
 
 Let $M\in \HIgtr(k)$ be a homotopy sheaf with generalized transfers. The presheaf $\tilde{\Gamma}_*(M)$ of abelian groups, defined by
 
 \begin{center}
 $\tilde{\Gamma}_*(M)(X)=A^0(X,M\otimes(\LLL_{X/k})^{\vee})$
 
 \end{center}
 for any smooth scheme $X/k$, is a MW-homotopy sheaf (see Definition \ref{DefMWhomotopySheaves}) canonically isomorphic to $M$ as presheaves. 
 \end{The}
 
\begin{proof}
This follows from Theorem \ref{HIgtrToHIMW} and the fact that the natural map $M\to \tilde{\Gamma}_*(M)$ agrees on the stalk level (see \cite[Theorem 5.41]{Mor12}).
\end{proof}

\begin{Par} \label{HIgtrFunctor}
A morphism of sheaves with generalized transfers commutes with the transfers $\Tr_{F/E}$ and the $GW$-action hence induces a natural transformation on the Rost-Schmid complex which commutes with the respective maps \ref{RostSchmidPushforward}, \ref{RostSchmidPullback} and \ref{RostSchmidGWaction} defined on the Rost-Schmid complex. After a careful examination of the proof of Theorem \ref{HIgtrToHIMW}, we can thus define a functor 
$$
\begin{array}{rcl}
\tilde{\Gamma}_*:\HIgtr(k) & \to & \HIMW(k) \\
M & \mapsto & \tilde{\Gamma}_*(M)
\end{array}
$$
which is conservative.

\end{Par}

We end this section with a lemma that will be useful later.
\begin{Lem} \label{LemHIgtr}
Let $M\in \HIgtr(k)$ be a homotopy sheaf with generalized transfers and let $\psi:E\to F$ be a finite extension of fields. For any finite model\footnote{More precisely, $X$ (resp. $Y$) is an irreducible smooth scheme with function field $E$ (resp. $F$) and the map $f:Y\to X$ is finite.} $f:Y\to X$ of $\psi$, we have defined in \ref{RostSchmidPushforward} a pushforward map
\begin{center}
$f_*:A^ 0(Y,M)\to A^0(X,M)$.
\end{center}
The limit of all such maps over finite models $f:Y\to X$ defines a map
\begin{center}
$M(F,\LLL_{F/k})\to M(E,\LLL_{E/k})$
\end{center}
which is equal to the generalized transfer map $\Tr_{F/E}$.

\end{Lem}
\begin{proof}
This follows from the definitions.
\end{proof}

\section{Morel's conjecture on Bass-Tate transfers}
\label{SectionMorelConjecture}

\subsection{Bass-Tate transfers} \label{SubsectionBassTateTransfers}

\begin{Par} \label{gtrMoinsUn}
Let $M$ be a homotopy sheaf and $M_{-1}$ its contraction. We recall the construction of the Bass-Tate transfer maps
\begin{center}
$\tr_{\psi}=\tr_{F/E}:M_{-1}(F,\LLL_{F/k}) \to M_{-1}(E,\LLL_{E/k})$
\end{center}
defined for any finite map $\psi:E\to F$ of fields.
\end{Par}

\begin{The}
Let $M\in \HI(k)$ be a homotopy sheaf. Let $F$ be a field and $F(t)$ the field of rational fractions with coefficients in $F$ in one variable $t$. We have a split short exact sequence

\begin{center}
$\xymatrix@C=10pt@R=20pt{
0 \ar[r] &  M(F) \ar[r]^-{\res}   & M(F(t)) \ar[r]^-d  &  \bigoplus_{x\in {(\AAA_F^1)}^{(1)}} \M(\kappa(x), \LLL_{x}) \ar[r] & 0
}
$
\end{center}
where $d=\bigoplus_{x\in {(\AAA_F^1)}^{(1)}}\partial_x$ is the usual differential (see \ref{e2.0.1}).
\end{The}
 \begin{proof}
See \cite[Theorem 5.38]{Mor12}\footnote{In fact, Morel does not use twisted sheaves but chooses a canonical generator for each $\LLL_x$ instead, which is equivalent.}.
\end{proof}

\begin{Def}[Coresidue maps]
Keeping the previous notations, the fact that the homotopy sequence is (canonically) split allows us to define {\em coresidue maps}
\begin{center}
$\rho_x:\M(\kappa(x),\LLL_{x})\to M(F(t))$
\end{center}
for any closed points $x\in {(\AAA_F^1)}^{(1)}$, satisfying $\partial_x\circ \rho_x=\Id_{\kappa(x)}$ and $\partial_y\circ \rho_x=0$ for $x\neq y$ where $y\in {(\AAA_F^1)}^{(1)}\cup \{\infty\}$.
\end{Def}
\begin{Def}[Bass-Tate transfers] \label{BassTateTransfersDefinition} 
Let $M\in \HI(k)$ be a homotopy sheaf. Let $F$ be a field and $F(t)$ the field of rational fractions with coefficients in $F$ in one variable $t$. For $x\in {(\AAA_F^1)}^{(1)}$, we define the Bass-Tate transfer
\begin{center}
$\tr_{x/F}:\M(F(x),\LLL_{F(x)/k})\to \M(F,\LLL_{F/k})$ 

\end{center}
by the formula $\tr_{x/F}=-\partial_{\infty}\circ \rho_x$.

\end{Def}

\begin{Rem}
There is also an equivalent definition of the Bass-Tate transfers that does not use the coresidue maps (see \cite[§4.2]{Mor12}.
\end{Rem}

\begin{Lem} \label{R3a}
Let $M\in \HI(k)$ be a homotopy sheaf. Let $\phi:E\to F$ be a field extension which induces a morphism, and $w$ be a valuation on $F$ which restricts to a non trivial valuation $v$ on $E$ with ramification $e$. We have a commutative square
\begin{center}
$\xymatrix{
M(E)
\ar[r]^-{\partial_v}
\ar[d]_{\phi_*}
&
\M(\kappa(v),\LLL_v)
\ar[d]^{e_{\epsilon}\cdot \bar{\phi}_*}
\\
M(F)
\ar[r]_-{\partial_w}
&
\M(\kappa(w),\LLL_w)
}$
\end{center} 
where $\overline{\phi}:\kappa(v)\to \kappa(w)$ is the induced map and $e_{\epsilon}=\sum_{i=1}^{e}\ev{-1}^{i-1}$.

\end{Lem}
\begin{proof}
See \cite[§3]{Fel19}
\end{proof}

We now prove a base change formula (see also \cite[Claim 3.10]{Fel18} for a similar result). The proof is similar to the original case where $M$ represents Milnor K-theory (see \cite[§1]{BassTate73}, or \cite[§7]{GS17}).

\begin{Lem}\label{R1c_faible}
Let $M\in \HI(k)$ be a homotopy sheaf. Let $F/E$ be a field extension and $x\in {(\AAA_E^1)}^{(1)}$ a closed point. Then the following diagram
\begin{center}
$\xymatrixcolsep{5pc}\xymatrix{
\M(E(x),\LLL_{E(x)/k}) \ar[r]^-{\tr_{x/E}}   \ar[d]_{\oplus_y \res_{F(y)/E(x)}} 
&
 \M(E,\LLL_{E/k}) \ar[d]^{\res_{F/E}} 
\\
\bigoplus_{y\mapsto x} \M(F(y),\LLL_{F(y)/k}) \ar[r]_-{\sum_{y}e_{y,\epsilon}\tr_{y/F}} & \M(F,\LLL_{F/k}) 
}$
\end{center}
is commutative, where  the notation $y\mapsto x$ stands for the closed points of $ {(\AAA_F^1)}$ lying above $x$, and $e_{y,\epsilon}=\sum_{i=1}^{e_y}\ev{-1}^{i-1}$ is the quadratic form associated to the ramification index of the valuation $v_y$ extending $v_x$ to $F(t)$.
\end{Lem}

\begin{proof}
According to Lemma \ref{R3a}, the following diagram
\begin{center}
$\xymatrix{
M(E(t)) \ar[r]^{\partial_x}  \ar[d]_{\res_{F(t)/E(t)}}
& \M(E(x),\LLL_x) \ar[d]^{\oplus_y e_{y,\epsilon} \res_{F(y)/E(x)}} 
\\
M(F(t)) \ar[r]_-{\oplus_y \partial_y} & \bigoplus_{y\mapsto x} \M(F(y),\LLL_y)
}$
\end{center} 
is commutative hence so does the diagram
\begin{center}
$\xymatrix{
M(E(t))   \ar[d]_{\res_{F(t)/E(t)}}
& \M(E(x),\LLL_x) \ar[d]^{\oplus_y e_{y,\epsilon} \res_{F(y)/E(x)}} \ar[l]_{\rho_x} \\
M(F(t))  &
 \bigoplus_{y\mapsto x} \M(F(y),\LLL_y)\ar[l]^-{\oplus_y \rho_y}
}$
\end{center} 
since the difference is killed by all $\partial_y$ for $y\in {(\AAA_F^1)}^{(1)}\cup \{\infty\}$. Then, we conclude according to the definition of the Bass-Tate transfer maps \ref{BassTateTransfersDefinition}.
\end{proof}

\begin{Rem}
The multiplicities $e_y$ appearing in the previous lemma are equal to 
\begin{center}
${[E(x):E]_i/[F(y):F]_i}$
\end{center} 
where $[E(x):E]_i$ is the inseparable degree.
\end{Rem}

\begin{Lem}\label{R2bMoinsUn}
Let $M\in \HI(k)$ be a homotopy sheaf. Let $\phi:E\to F=E(x)$ be a simple extension. Then
\begin{enumerate}

\item For $\ev{a}\in \GW(E)$ and $\mu \in M(F,\LLL_{F/k})$, one has	$\tr_{x/E}\big(\ev{\psi(a)}\cdot \mu\big)=\ev{a}\cdot \tr_{x/E}(\mu)$.

\item For $\ev{a}\in \GW(F,\LLL_{F/k})$ and $\mu \in M(E)$, one has 	$\tr_{x/E}\big(\ev{a}\cdot \res_{F/E}(\mu)\big)=\tr_{x/E}(\ev{a})\cdot \mu$.

\end{enumerate} 
\end{Lem}
\begin{proof}
This follows by $\operatorname{GW}$-linearity from the definitions (see \cite[Lemma 5.24]{BachmannYakerson18}).
\end{proof}

\begin{Def}
Let $M\in \HI(k)$ be a homotopy sheaf. We denote by $M_{\QQ}$ the homotopy sheaf defined by
\begin{center}
$X\mapsto M(X)\otimes_{\ZZ} \QQ$
\end{center}
and by $M_{-1,\QQ}$ the homotopy sheaf $(M_{\QQ})_{-1}$.
\end{Def}

\begin{Cor} \label{Lem1.8}
Let $M\in \HI(k)$ be a homotopy sheaf. Let $\phi:E\to F=E(x)$ be a simple extension. The kernel of the restriction map $\res_{F/E}:\M(E)\to \M(F)$ is killed by $\tr_{x/E}(1)$. In particular: if the Hopf map $\eeta$ acts trivially on $\M$, then the restriction map $\res_{F/E}:\MQ(E)\to \MQ(F)$ is injective.
\end{Cor}
\begin{proof}
This follows from the previous projection formula and the fact that, if $ \eeta$ acts trivially, then the action of $\tr_{x/E}(1)$ on $\MQ(E)$ is the multiplication by a nonzero natural number.
\end{proof}

\begin{Def} \label{DefGeometricTransfers}
Let $F=E(x_1,x_2,\dots, x_r)$ be a finite extension of a field $E$ and consider the chain of subfields
\begin{center}

$E\subset E(x_1)\subset E(x_1,x_2)\subset \dots \subset E(x_1,\dots , x_r)=F.$
\end{center}
Define by induction:
\begin{center}
$\tr_{x_1,\dots,x_r/E}=\tr_{x_r/E(x_1,\dots,x_{r-1})}\circ \dots \circ \tr_{x_2/E(x_1)}\circ \tr_{x_1/E}$
\end{center}
\end{Def}

\begin{Conj}[Morel conjecture] \label{ConjectureMoinsUn}
Keeping the previous notations, the maps $\tr_{x_1,\dots,x_r/E}:\M(F,\LL_{F/k})\to \M(E,\LL_{E/k})$ do not depend on the choice of the generating system $(x_1,\dots , x_r)$.

\end{Conj}
\begin{proof}
This was claimed by Morel in \cite[Remark 4.31]{Mor12} and \cite[Remark 5.10]{Mor11} (see also \cite[Remark 4.3]{Bach20} for a similar conjecture).
\end{proof}
\begin{Pro}[Projection formulas] \label{ContractionProjectionFormulas} Let $\phi:E\to F=E(x_1,x_2,\dots, x_r)$ be a finite extension. Then
\begin{enumerate}

\item For $\ev{a}\in \GW(E)$ and $\mu \in M(F,\LLL_{F/k})$, one has 	$\tr_{x_1,\dots ,x_r/E}(\ev{\psi(a)}\cdot \mu)=\ev{a}\cdot \tr_{x_1,\dots ,x_r/E}(\mu)$.

\item For $\ev{a}\in \GW(F,\LLL_{F/k})$ and $\mu \in M(E)$, one has 	$\tr_{x_1,\dots ,x_r/E}(\ev{a}\cdot \res_{F/E}(\mu))=\tr_{x_1,\dots ,x_r/E}(\ev{a})\cdot \mu$.

\end{enumerate} 

\end{Pro}

\begin{proof} This is immediate by induction on $r$ according to Lemma \ref{R2bMoinsUn}.
\end{proof}

\begin{The} [Strong R1c] \label{R1c_fort} Let $M\in \HI(k)$ be a homotopy sheaf. Let $E$ be a field, $F/E$ a finite field extension and $L/E$ an arbitrary field extension. Write $F=E(x_1,\dots, x_r)$ with $x_i\in F$, $R=F\otimes_E L$ and $\psi_{\mathfrak{p}}:R\to R/\mathfrak{p}$ the natural projection defined for any $\mathfrak{p}\in \Spec(R)$. Then the diagram
\begin{center}
$\xymatrixcolsep{9pc}\xymatrix{
\M(F,\LLL_{F/k}) \ar[r]^{\tr_{x_1,\dots, x_r/E}} \ar[d]_{\oplus_{\mathfrak{p}} \res_{(R/\mathfrak{p})/F}}& 
\M(E,\LLL_{E/k})  \ar[d]^{\res_{L/E}} \\
\bigoplus_{\mathfrak{p}\in \Spec(R)}
\M(R/\mathfrak{p},\LLL_{(R/\mathfrak{p})/k}) 
\ar[r]_-{
\sum_{\mathfrak{p}}  (m_{p})_{\epsilon} \tr_{
\psi_{\mathfrak{p}}(a_1),
\dots,
\psi_{\mathfrak{p}}(a_r)/L}}
& \M(L,\LLL_{L/k})
}$
\end{center}
is commutative where $(m_{p})_{\epsilon}$ the quadratic form associated to the length of the localized ring $R_{(\mathfrak{p})}$ (see Notation \ref{Notation}.
\end{The}

\begin{proof}
We prove the theorem by induction. For $r=1$, this is Lemma \ref{R1c_faible}. Write $E(x_1)\otimes_E L=\prod_j R_j$ for some Artin local $L$-algebras $R_j$, and decompose the finite dimensional $L$-algebra $F\otimes_{E(x_1)} R_j$ as $F\otimes_{E(x_1)} R_j=\prod_{i} R_{ij}$ for some local $L$-algebras $R_{ij}$. We have $F\otimes_E L\simeq \prod_{i,j}R_{ij}$. Denote by $L_j$ (resp. $L_{ij}$) the residue fields of the Artin local $L$-algebras $R_j$ (resp. $R_{ij}$), and $m_j$ (resp. $m_{ij}$) for their geometric multiplicity. one can conclude as the following diagram commutes

\begin{center}
\scalebox{0.88}{$\xymatrixcolsep{10.5pc}\xymatrix{
\M(F,\LLL_{F/k}) \ar[r]^{\tr_{x_1,\dots, x_r/E}} \ar[d]_{\oplus_{ij} \res_{L_{ij}/F}}
&
\M(E(x_1),\LLL_{E(x_1)/k}) 
\ar[d]^{\oplus \res_{L_j/E(x_1)}}
\ar[r]^{\tr_{x_1/E}}
& \M(E,\LLL_{E/k})  
\ar[d]^{\res_{L/E}} 
\\
\bigoplus_{ij}\M(L_{ij},\LLL_{L_{ij}/k}) 
\ar[r]_-{
\sum_{ij}  (m_{ij}m_j^{-1})_{\epsilon} \tr_{
\psi_{ij}(x_1),
\dots,
\psi_{ij}(x_r)/L_j}}
&
\bigoplus_j \M(L_j,\LLL_{L_j/k}) 
\ar[r]_{\sum_j (m_j)_{\epsilon}\tr_{\psi_j(x_1)/L}}
& \M(L,\LLL_{L/k})
}$
}

\end{center}
since both squares are commutative by the inductive hypothesis and the multiplicity formula $(mn)_{\epsilon}=m_{\epsilon}n_{\epsilon}$ for any natural numbers $m,n$.

\end{proof}

\begin{Pro} \label{Lem1.11}
Let $M\in \HI(k)$ be a homotopy sheaf. Let $E\to F$ a field extension. If the Hopf map $\eeta$ acts trivially on $\M$, then the restriction map $\res_{F/E}:\MQ(E)\to \MQ(F)$ is injective.

\end{Pro}
\begin{proof} We may assume that $F/E$ is finitely generated. By induction on the number of generators of $F$ over $E$, we may assume that $F=E(x)$ is generated by a single element $x\in F$. If $x$ is algebraic over $E$, we know from Corollary \ref{Lem1.8} that $\res_{F/E}$ is killed by $\Tr_{x/E}(1)$. Since the Hopf map acts trivially, the action of $\Tr_{x/E}(1)$ is given by the multiplication by $[F:E]$ and we obtain the result. If $x$ is transcendent over $E$, then we know that $\M(E)$ is a direct summand in $\M(F)$.
\end{proof}

\begin{Cor} \label{TransfersRatioPartiePlus}
Let $M\in \HI(k)$ be a homotopy sheaf. Let $F/E$ be a finite field extension, and let $x_1,\dots,x_r$ and $y_1,\dots, y_s$ be two generating system for $F/E$. If the Hopf map $\eeta$ acts trivially on $\M$, then 
\begin{center}
$\tr_{x_1,\dots,x_r/E}=\tr_{y_1,\dots,y_s/E}$
\end{center}
seen as morphism from $\MQ(F)$ to $\MQ(E)$.
\end{Cor}
\begin{proof}
We apply Theorem \ref{R1c_fort} two times with $F=E(x_1,\dots,x_r)$ and $F=E(y_1,\dots, y_s)$, and with $L=\overline{E}$ an algebraic closure of $E$. Hence $\res_{L/E}\circ\tr_{x_1,\dots,x_r/E}=\res_{L/E}\circ\tr_{y_1,\dots,y_s/E}$ and we end with Proposition \ref{Lem1.11}.
\end{proof}

\begin{The}
Assume that $2$ is invertible. Let $M\in \HI(k)$ be a homotopy sheaf. The sheaf $\MQ$ has functorial Bass-Tate transfer maps in the sense of Corollary \ref{TransfersRatioPartiePlus}.

\end{The}
\begin{proof}
The sheaf $\MQ$ splits into two sheaves $\MQ^+$ and $\MQ^-$. On one hand, the Hopf map $\eeta$ acts trivially on $\MQ^+$ hence there is a structure of functorial transfers thanks to Corollary \ref{TransfersRatioPartiePlus}. On the other hand, the Hopf map $\eeta$ is invertible on the minus part, leading to an isomorphism between $\MQ$ and $(\MQ)_{-1}$. The latter has a structure of functorial transfers according to \cite[Chapter 4]{Mor12}. Hence the result. 
\end{proof}

We summarize the previous results in the following theorem.
\begin{The} \label{HIMoinsUntogtr}
Let $M\in \HI(k)$ be a homotopy sheaf. Then:
\begin{enumerate}
\item Assume that $2$ is invertible. The rational contracted homotopy sheaf $M_{-1,\QQ}$ is a homotopy sheaf with generalized transfers.
\item Assuming Conjecture \ref{ConjectureMoinsUn}, the contracted homotopy sheaf $\M$ is a homotopy sheaf with generalized transfers.
\end{enumerate}
\end{The}

\begin{proof}
The $\GW$-action is defined in \ref{GWstructureContracted}, functoriality of transfers \ref{itm:eR1b} follows from Corollary \ref{TransfersRatioPartiePlus} or Conjecture \ref{ConjectureMoinsUn}, the base change \ref{itm:eR1c} is Theorem \ref{R1c_fort}, the projection formulas \ref{itm:eR2} are proved in Proposition \ref{ContractionProjectionFormulas} and the compatibility axiom \ref{itm:eR3b} can be deduced from \cite[Theorem 5.19]{Mor12}.

\end{proof}

\subsection{Unicity of transfers} \label{SubsectionUnicityTransfers}

The goal of this subsection is to prove Theorem \ref{UnicityTransfers} which asserts that the structure of Bass-Tate transfer maps on a contracted homotopy sheaf $\M$ is, in some sense, unique.

\begin{Lem}[eR3c and eR3d] \label{eR3cdContraction} Let $M\in \HI(k)$ be a homotopy sheaf. Let $\iota:E\to F$ be an extension of fields and $w$ a valuation on $F$ that restricts to the trivial valuation on $E$. Then the composition
\begin{center}
$\xymatrix{
M(E)
\ar[r]^{\iota^*}
&
M(F)
\ar[r]^-{\partial_w}
&
M_{-1}(\kappa(w),\LLL_w)
}$
\end{center}
is zero.
Moreover, let $\bar{\iota}:E\to \kappa(w)$ be the induced map. For any prime $\pi$ of $w$, the following diagram
\begin{center}
$\xymatrix{
M_{-1}(E)
\ar[r]^{\iota^*}
\ar[d]^{\bar{\iota}^*}
&
M_{-1}(F)
\ar[d]^{[\pi]}
\\
\M(\kappa(w))
&
M(F)
\ar[l]^-{\Theta\circ \partial_w}
}$
\end{center}
is commutative (where $\Theta$ is the canonical isomorphism induced by the trivialization of $\LLL_{w}$ through the choice of $\pi$).
\end{Lem}
\begin{proof}
The first identity is deduced from the long exact sequence \ref{HILongExactSequence}. The commutative square follows from \cite[Corollary 5.23]{BachmannYakerson18}.
\end{proof}

\begin{Lem}[eR3e] \label{eR3eContraction}
Let $E$ be a field over $k$ with a valuation $v$ and $u$ a unit of $v$. Then
\begin{center}
$\partial_v\circ [u]=\epsilon [u] \circ \partial_v$
\end{center}
where $\epsilon$ is $-\ev{-1}$ as usual.
\end{Lem}
\begin{proof}
See \cite[Lem 5.10]{Mor12}.

\end{proof}

\begin{The} \label{UnicityTransfers}
Let $M\in \HIgtr(k)$ be a homotopy sheaf with generalized transfers. We have for any simple extension $\psi: F\to F(x)$:
\begin{description}
\item[(a)] A generalized transfer map $\Tr_{\psi}:\M(F(x),\LLL_{F(x)/k})\to \M(F,\LLL_{F/k})$ induced by the structure of a homotopy sheaf with generalized transfers on $M$ (see also Remark \ref{GeneralizedTransfersOnContraction}).
\item[(b)] A Bass-Tate transfer map $\tr_{x/F}:\M(F(x),\LLL_{F(x)/k})\to \M(F,\LLL_{F/k})$ defined in \ref{BassTateTransfersDefinition}.
\end{description}
Then the two transfer maps coincide : $\Tr_{\psi}=\tr_{x/F}$.
\end{The}\label{UniqueTransfers}
\begin{proof}
Fix a field $F$ and $F(t)$ the field of rational fractions with coefficients in $F$ in one variable $t$, and fix a simple extension $\psi: F\to F(x)$. Consider the canonical inclusion $\iota:F(x)\to F(x)(t)$ and define
\begin{center}
$\Phi^x:\M(F(x),\LLL_{F(x)/k})\to M(F(t),\LLL_{F(t)/k})$ \\
\end{center}
as the composite $\Phi^x=\Tr_{F(x)(t)/F(t)}\circ [t-\iota(x)] \circ \iota^*$ where $\Tr$ denotes the generalized transfers of point \textbf{(a)}. A combination of \ref{itm:eR3b}, Lemma \ref{eR3cdContraction} and Lemma \ref{eR3eContraction} shows that
\begin{align*}
\partial_x \circ \Phi^x &= \Id_{M(F(t),\LLL_{F(t)/k})}, \\
-\partial_{\infty} \circ \Phi^x&= \Tr_{F(x)/F},
\end{align*} 
which is exactly the definition of the Bass-Tate transfers $\tr_{F(x)/F}$.
\end{proof}

\section{MW-homotopy sheaves} \label{SectionMWHomotopySheaves}

\subsection{Sheaves with MW-transfers} \label{SubsectionSheavesWithMWTransfers} \label{MWRecollection}
In this subsection, we recall the basic definition of sheaves with MW-transfers in order to fix the notations. We follow the presentation of \cite[Chapter 2]{BCDFO}.
\begin{Par}
Let $X$ and $Y$ be smooth schemes over $k$ and let $T\subset X\times Y$ be a closed subset. Any irreducible component of $T$ maps to an irreducible component of $X$ through the projection $X\times Y\to X$ (in other words, the canonical projection map from $T$ to $X$ is finite equidimensional).
\end{Par}
\begin{Def} \label{DefAdmissiblSet}
If, when $T$ is endowed with its reduced structure, this map is finite and surjective for every irreducible component of $T$, we say that $T$ is an admissible subset of $X\times Y$. We denote by $\mathcal{A}(X,Y)$ the set of admissible subsets of $X\times Y$, partially ordered by inclusions.
\end{Def}
\begin{Par}
If $Y$ is equidimensional, $d=\dim Y$ and $p_Y:X\times Y\to Y$ is the projection, we define a covariant functor
\begin{center}
$\mathcal{A}(X,Y)\to \Ab$
\end{center}
by associating to each admissible subset $T\in \mathcal{A}(X,Y)$ the group $\CHt^d_T(X\times Y,p^*_Y\LLL_{Y/k})$ and to each morphism $T'\subset T$ the extension of support morphism
\begin{center}
$\CHt^d_{T'}(X\times Y,p^*_Y \LLL_{Y/k})\to \CHt^d_T(X\times Y,p^*_Y\LLL_{Y/k})$
\end{center}
and, using that functor, we set
\begin{center}
$\Cortilde_k(X,Y)=\colim_{T\in \mathcal{A}(X,Y)}\CHt^d_T(X\times Y,p^*_Y\LLL_{Y/k})$.
\end{center}
If $Y$ is not equidimensional, then $Y=\bigsqcup_j Y_j$ with each $Y_j$ equidimensional and we set
\begin{center}
$\Cortilde_k(X,Y)=\prod_j \Cortilde_k(X,Y_j)$.
\end{center}
By additivity of Chow-Witt groups, if $X=\bigsqcup_i X_i$ and $Y=\bigsqcup_jY_j$ are the respective decompositions of $X$ and $Y$ in irreducible components, we have
\begin{center}
$\Cortilde_k(X,Y)=\prod_{i,j} \Cortilde_k(X_i,Y_j)$.
\end{center}
\end{Par}
\begin{Rem}
In the sequel, we will simply write $\LLL_Y$ in place of $p^*_Y\LLL_{Y/k}$.
\end{Rem}

\begin{Exe} \label{MWtransfersExe}
Let $X$ be a smooth scheme of dimension $d$. Then
\begin{center}
$\Cortilde_k(\Spec(k),X)=\bigoplus_{x\in X^{(d)}} \CHt^d_{\{x\}}(X,\LLL_X)=\bigoplus_{x\in X^{(d)}} \GW(\kappa(x),\LLL_{\kappa(x)/k})$.
\end{center}
On the other hand, $\Cortilde_k(X,\Spec(k))=\CHt^0(X)=\kMW_0(X)$ for any smooth scheme $X$.
\end{Exe}
\begin{Par}
The group $\Cortilde_k(X,Y)$ admits an alternate description which is often useful. Let $X$ and $Y$ be smooth schemes, with $Y$ equidimensional. For any closed subscheme $T\subset X\times Y$ of codimension $d=\dim Y$, we have an inclusion
\begin{center}
$\CHt^d_T(X\times Y,\LLL_{Y/k})\subset \bigoplus_{x\in (X\times Y)^{(d)}}\kMW_0(\kappa(x),\det(\LLL_{x}\otimes (\LLL_{Y/k})_x))$
\end{center}
and thus
\begin{center}
$\Cortilde_k(X,Y)=\bigcup_{T\in \mathcal{A}(X,Y)} \CHt^d_T(X\times Y, \LLL_{Y/k}) \subset \bigoplus_{x\in (X\times Y)^{(d)}}\kMW_0(\kappa(x),\det(\LLL_{x}\otimes (\LLL_{Y/k})_x))$. 
\end{center}
In general, the inclusion $\Cortilde_k(X,Y) \subset \bigoplus_{x\in (X\times Y)^{(d)}}\kMW_0(\kappa(x),\det(\LLL_{x}\otimes (\LLL_{Y/k})_x))$ is strict as shown by Example \ref{MWtransfersExe}. As an immediate consequence of this description, we see that the map
\begin{center}
$\CHt^d_T(X\times Y,\LLL_Y)\to \Cortilde_k(X,Y)$
\end{center}
is injective for any $T\in \mathcal{A}(X,Y)$.
\end{Par}

\begin{Par}[Composition of finite MW-correspondences]
Let $X,Y$ and $Z$ be smooth schemes of respective dimension $d_X,d_Y$ and $d_Z$, with $X$ and $Y$ connected. Let $V\in \mathcal{A}(X,Y)$ and $\mathcal{A}(Y,Z)$ be admissible subsets. If $\beta\in \CHt^{d_Y}_V(X\times Y,\LLL_{Y/k})$ and $\alpha \in \CHt^{d_Z}_T(Y\times Z,\LLL_{Z/k})$ are two cycles, then the expression
\begin{center}
$\alpha \circ \beta = (q_{XY})_*[(q_{YZ})^*\beta \cdot  (p_{XY})^*\alpha]$
\end{center}
is well-defined and yields a composition
\begin{center}
$\circ: \Cortilde_k(X,Y)\times \Cortilde_k(Y,Z) \to \Cortilde_k(X,Z)$
\end{center}
which is associative.
\end{Par}

\begin{Def}
The category of finite MW-correspondences over $k$ is by definition the category $\Cortilde_k$ whose objects are smooths schemes and whose morphisms are the abelian groups $\Cortilde_k(X,Y)$.
\end{Def}

\begin{Rem}
The category $\Cortilde_k$ is a symmetric monoidal additive category (see \cite[Chapter 2, Lemma 4.4.2]{BCDFO}).
\end{Rem}

\begin{Def}\label{DefMWhomotopySheaves}
A presheaf with MW-transfers is a contravariant additive functor $\Cortilde_k \to \Ab$. A (Nisnevich) sheaf with MW-transfers is a presheaf with MW-transfers such that its restriction to $\Sm$ via the graph functor is a Nisnevich sheaf. We denote by $\widetilde{\operatorname{PSh}}(k)$ (resp. $\widetilde{\operatorname{Sh}}(k)$) the category of presheaves (resp. sheaves) with MW-transfers and by $\HIMW(k)$ the category of homotopy sheaves with MW-transfers (also called {\em MW-homotopy sheaves}).
\end{Def}

\begin{Exe}
For any $j\in \ZZ$, the contravariant functor $X\mapsto \kMW_j(X)$ is a presheaf on $\Cortilde_k$.
\end{Exe}

\begin{Par}{\sc Pushforwards} \label{MWPushforward}
Let $X$ and $Y$ be two smooth schemes of dimension $d$ and let $f:X\to Y$ be a finite morphism such that any irreducible component of $X$ surjects to the irreducible component of $Y$ it maps to. Assume that we have an orientation $(\LL,\psi)$ of $\LLL_f$, that is an isomorphism $\psi: \LL\otimes \LL \to \LLL_f$ of line bundles. We define a finite correspondence $\alpha(f,\LL,\psi)\in \Cortilde_k(Y,X)$. Let $\gamma_f':X\to Y\times X$ be the transpose of the graph of $f$.  Since $X$ is an admissible subset, we have a transfer map
\begin{center}
$(\gamma_f')_*:\kMW_0(X,\LLL_f)\to \CHt_X^d(Y\times X,\LLL_{X/k})\to \Cortilde_k(Y,X)$.
\end{center}
The map $\psi$ yields an isomorphism $\kMW_0(X)\to \kMW_0(X,\LLL_f)$. We define the finite MW-correspondence $\alpha(f,\LL,\psi)$ as the image of $\ev{1}$ under the composite
\begin{center}
$\kMW_0(X)\to \kMW_0(X,\LLL_f)\to \CHt_X^d(Y\times X,\LLL_{X/k})\to \Cortilde_k(Y,X)$.
\end{center}
\par Now let $M\in \HIMW(k)$ be a homotopy sheaf with MW-transfers. Denote by $(M\otimes \LLL_f)_X$ (resp. $M_Y$) the canonical (twisted) sheaf associated to $M$ defined on the Zariski site $X_{\txt{Zar}}$ (resp. $Y_{\txt{Zar}}$) introduced in \ref{DefMtwisted} and define a natural transformation
\begin{center}
$f_*: f_*(M\otimes \LLL_f)_X\to M_Y$
\end{center}
by taking (as in \ref{DefMtwisted}) the sheafification of the natural transformation of presheaves
\begin{center}
$V\in Y_{\txt{Zar}} \mapsto 
(M(f^{-1}(V),\LLL_{f_{|f^{-1}(V)}})\to M(V))$\\
$\,\,\,\,\,\,\,\,\,\,\,\,\,\,
(\mu\otimes l)\mapsto \alpha(f,\psi_l,L_l)^*(\mu)$
\end{center}
where $(\psi_l,\LL_l)$ is the orientation of $\LLL_{f_{|f^{-1}(V)}}$ associated to $l\in \LLL_{f_{|f^{-1}(V)}}^\times$. Taking global sections, this leads in particular to a map
\begin{center}
$M([{}^tf]):M(X,\LLL_f)\to M(Y)$
\end{center}
for any finite morphism $f:X\to Y$.

\par We can check the following propositions.
\end{Par}

\begin{Pro} \label{FunctorialityMWPushforward}
Let $M\in \HIMW(k)$ and consider two finite morphisms
$\xymatrix{
X
\ar[r]^f &
Y
\ar[r]^g &
Z}$
of smooth schemes. Then
\begin{center}
$M([{}^t{(g\circ f)}])=M([{}^tg])\circ M([{}^tf])$.
\end{center}
\end{Pro}
\begin{proof}
Keeping the previous notations, if $(\LL',\psi')$ is an orientation of $\LLL_g$, then $(\LL\otimes f^*\LL',\psi\otimes f^*\psi')$ is an orientation of $\LLL_{g\circ f}=\LLL_f\otimes f^*\LLL_g$, and we have $\alpha(f,\LL,\psi)\circ \alpha(g,\LL',\psi')=\alpha(g \circ f, \LL\otimes f^*\LL', \psi \otimes f^*\psi')$.
\end{proof}

\begin{Pro} \label{R3bMWtransfers}
Let $M\in \HIMW(k)$ be a homotopy sheaf with MW-transfers. Let $i:Z\to X$ and $i':T\to Y$ be two closed immersions and let $f:Y\to X$ be a finite morphism. The following diagram 
\begin{center}
$\xymatrix{
M(Y- T,\LLL_{(Y-T)/k})
\ar[r]^-{\partial}
\ar[d]_{M([{}^tf])}
&
\M(T,\LLL_{T/k})
\ar[d]^{M([{}^tf])}
\\
M(X-Z,\LLL_{(X-Z)/k})
\ar[r]^-{\partial}
&
\M(Z,\LLL_{Z/k})
}$
\end{center}
is commutative.
\end{Pro}

\begin{Par} {\sc Tensor products}
Let $X_1,X_2,Y_1,Y_2$ be smooth schemes over $\Spec k$. Let $d_1=\dim Y_1$ and $d_2=\dim Y_2$. Let $\alpha_1\in \CHt_{T_1}^{d_1}(X_1\times Y_1,\LLL_{Y_1/k})$ and $\alpha_2\in \CHt_{T_2}^{d_2}(X_2\times Y_2,\LLL_{Y_2/k})$ for some admissible subsets $T_i\subset X_i\times Y_i$. The exterior product defined in \cite[§4]{Fasel13} gives a cycle
\begin{center}
$(\alpha_1\times \alpha_2)\in \CHt^{d_1+d_2}_{T_1\times T_2}(X_1\times Y_1\times X_2\times Y_2,p^*_{Y_1}\LLL_{Y_1/k}\otimes p^*_{Y_2}\LLL_{Y_2/k})$
\end{center}
where $p_{Y_i}:X_1\times Y_1 \times X_2\times Y_2\to Y_i$ is the canonical projection to the corresponding factor. Let $\sigma:X_1\times Y_1\times X_2\times Y_2\to X_1\times X_2\times Y_1\times Y_2$ be the transpose isomorphism. Applying $\sigma_*$, we get a cycle 
\begin{center}
$\sigma_*(\alpha_1\times \alpha_2)\in \CHt^{d_1+d_2}_{\sigma(T_1\times T_2)}(X_1\times X_2\times Y_1\times Y_2,p^*_{Y_1}\LLL_{Y_1/k}\otimes p^*_{Y_2}\LLL_{Y_2/k})$.
\end{center}
Since $p^*_{Y_1}\LLL_{Y_1/k}\otimes p^*_{Y_2}\LLL_{Y_2/k})=\LLL_{Y_1\times Y_2/k}$, it is straightforward to check that $\sigma(T_1\times T_2)$ is finite and surjective over $X_1\times X_2$. Thus $\sigma_*(\alpha_1\times \alpha_2)$ defines a finite MW-correspondence between $X_1\times X_2$ and $Y_1\times Y_2$.
 
\end{Par}
\begin{Def}
Let $X_1,X_2,Y_1,Y_2$ be smooth schemes over $\Spec k$, and $\alpha_1\in \Cortilde_k(X_1,Y_1)$ and $\alpha_2\in \Cortilde_k(X_2,Y_2)$ two MW-correspondences. We define their tensor products as $X_1\otimes X_2=X_1\times X_2$ and $\alpha_1\otimes \alpha_2=\sigma_*(\alpha_1\times \alpha_2)$.
\end{Def}

\begin{Par} \label{MWTransfersMonoidal}
We denote by $\tilde{c}(X):Y\mapsto \Cortilde_k(Y,X)$ the representable presheaf associated to a smooth scheme $X$ (be careful that this is not a Nisnevich sheaf in general). The category of MW-presheaves is an abelian Grothendieck category with a unique symmetric monoidal structure such that the Yoneda embedding
\begin{center}
$\Cortilde_k\to \widetilde{\operatorname{Sh}}(k),X\mapsto \tilde{a}\tilde{c}(X)$
\end{center}
is symmetric monoidal (where $\tilde{a}$ is the sheafification functor, see \cite[Chapter 3, §1.2.7]{BCDFO}). The tensor product is denoted by $\otimes_{\HIMW}$ and commutes with colimits (hence the monoidal structure is closed, see \cite[Chapter 3, §1.2.14]{BCDFO}).

\end{Par}

\begin{Par} \label{GWstructureCompatibility}
Let $M\in \HIMW(k)$ be a homotopy sheaf with MW-transfers. The $GW$-module structure induced on $\M$ is compatible with the $GW$-module structure defined on $\M$ in \ref{GWstructureContracted}.
\end{Par}
\subsection{Structure of generalized transfers} \label{SubsectionStructureGeneralizedTransfers}

In this section, we study the category of MW-homotopy sheaves. For any MW-homotopy sheaf we construct a canonical structure of generalized transfers (see Definition \ref{DefGeneralizedTransfers}).
\begin{Par}
Let $M \in \HIMW(k)$ be a homotopy sheaf with MW-transfers. We denote by $\tilde{\Gamma}^*(M)$ the homotopy sheaf $M$ equipped with its structure of $\GW$-module coming from its structure of MW-transfers, and we define generalized transfers as follows. Let $\psi:E\to F$ be a finite extension of fields. Consider a smooth model $(X,x)$ (resp. $(Y,y)$) of $E/k$ (resp. $F/k$) such that $\psi$ corresponds to a map $Y_y\to X_x$. We may assume that this map is induced by a finite morphism $f:Y\to X$. We consider the pushforward on the MW--homotopy sheaf $\tilde \Gamma^*(M)$ with respect to the finite morphism $f$ defined in \ref{MWPushforward} and take the limit over all model of $F/E$ so that we obtain a morphism
\begin{center}
$\psi^*:\tilde{\Gamma}^*(M)(F,\LLL_{F/k})\to \tilde{\Gamma}^*(M)(E,\LLL_{E/k})$
\end{center}
of abelian groups. This defines a homotopy sheaf $\tilde{\Gamma}^*(M)$ canonically isomorphic to $M$ (as presheaves) and equipped with a structure of transfers \ref{itm:eD2}.
\end{Par}
\begin{The}
Keeping the previous notations, the transfer maps $\psi^*$ defines a structure of generalized transfers (see Definition \ref{DefGeneralizedTransfers}) on the homotopy sheaf $\tilde{\Gamma}^*(M)$.
\end{The}
\begin{proof}
The functoriality property \ref{itm:eR1b} results from Proposition \ref{FunctorialityMWPushforward}.

According to Proposition \ref{R2bMoinsUn} and Theorem \ref{R1c_fort}, the projection formulas \ref{itm:eR2} and the base change rule \ref{itm:eR1c} are true for any contracted homotopy sheaf hence we only have to prove that $\tilde{\Gamma}^*(M)$ is a contracted homotopy sheaf:
\begin{Lem} \label{LemHIMW}
Let $M\in \HIMW(k)$ a MW-homotopy sheaf. Then there is a canonical isomorphism
\begin{center}
$\tilde{\Gamma}^*(M)\simeq (\Gm \otimes_{\HIMW} \tilde{\Gamma}^*(M))_{-1}$
\end{center}
of homotopy sheaves which is compatible with the generalized transfers structure in the sense that the diagram
\begin{center}

$\xymatrix{
\tilde{\Gamma}^*(M)(E(x))
\ar[r]^{\psi^*}
\ar[d]^{\simeq}
&
\tilde{\Gamma}^*(M)(E)
\ar[d]^{\simeq}
\\
(\Gm \otimes_{\HIMW} \tilde{\Gamma}^*(M))_{-1}(E(x))
\ar[r]^{\tr_{x/E}}
&
(\Gm \otimes_{\HIMW} \tilde{\Gamma}^*(M))_{-1}(E)
}$
\end{center}
is commutative for any simple extension $\psi:E\to E(x)$ of fields, where $\tr_{x/E}$ is the Bass-Tate transfer map defined in \ref{BassTateTransfersDefinition}.
\end{Lem}

\begin{proof}
The isomorphism $\tilde{\Gamma}^*(M)\simeq (\Gm \otimes_{\HIMW} \tilde{\Gamma}^*(M))_{-1}$ is an equivalent reformulation of the cancellation theorem \cite[Theorem 4.0.1]{FasOst17}. The second assertion is a corollary of Theorem \ref{UnicityTransfers}.
\end{proof}

Still denoting by $M\in \HIMW(k)$ a MW-homotopy sheaf, we need to check that $\tilde{\Gamma}^*(M)$ satisfies \ref{itm:eR3b} where $M\in \HIMW(k)$, which deduced from the definitions and Proposition \ref{R3bMWtransfers}.

\end{proof}

\begin{Par}
As in \ref{HIgtrFunctor}, we see that $\tilde{\Gamma}^*$ defines a functor 
\begin{center}

$\tilde{\Gamma}^*:\HIMW(k) \to \HIgtr(k)$ \\
$M\mapsto \tilde{\Gamma}^*(M)$
\end{center}
which is conservative.

\end{Par}

\begin{The} \label{EquivalenceHIMWandHIgtr}
Keeping the previous notations, the functors
\begin{center}
$\xymatrix{
\HIMW(k)
\ar@<1ex>[r]^{\tilde{\Gamma}^*}
&
\HIgtr(k)
\ar@<1ex>[l]^{\tilde{\Gamma}^*}
}$
\end{center}
form an equivalence of categories.
\end{The}

\begin{proof}
\par
First, let $M\in \HIgtr(k)$ be a homotopy sheaf with generalized transfers. For any smooth scheme $X$, we have a canonical isomorphism
\begin{center}
$a_X:\tilde{\Gamma}^*\tilde{\Gamma}_*(M)(X)\to M(X)$
\end{center}
which is compatible with pullback maps and the $\GW$-action. Compatibility with the generalized transfers \ref{itm:eD2} results from Lemma \ref{LemHIgtr}.

\par Second, let $M\in \HIMW(k)$ be a MW-homotopy sheaf. For any smooth scheme $X$, we have a canonical isomorphism
\begin{center}
$b_X:M \to \tilde{\Gamma}_*\tilde{\Gamma}^*(M)(X)$
\end{center}
which is compatible with (smooth) pullbacks and the $\GW$-action. Since pushforward $p_*$ of a finite map $p:Y\to X$ is locally given by the multiplication by the correspondence $\alpha(p,\psi_l,L_l)$ of \ref{MWPushforward}, we see that $b$ commutes with $p_*$. Thus $b$ commutes with the multiplication by any  cycle $\alpha \in \CHt^{d_Y}_T(X\times Y,\LLL_{Y/k})$ (where $X,Y$ are two smooth schemes and $T\in X\times Y$ is an admissible subset) thanks to the identity
\begin{center}
$\alpha^*(\beta)=(p_{X})_*(\alpha\cdot p_Y^{*}(\beta))$
\end{center}
where $p_X:T\to X\times Y$ and $p_Y:X\times Y \to  Y$ are the canonical morphisms.

\end{proof}

\section{Applications} \label{SectionApplications}

\subsection{Infinite suspensions of homotopy sheaves}

 In order to fix notations, we recall that we have the following commutative diagram of categories :
\begin{center}
$\xymatrix{
\Hp(k)
\ar@<1ex>[r]^-{\Sigma^{\infty}_{S^1}}
&
\SHS(k)
\ar@<1ex>[l]^-{\Omega^{\operatorname{\infty}}_{S^1}}
\ar@<1ex>[r]^{\Sigma^{\infty}_{\Gm}}
\ar@<1ex>[d]^N
&
\SH(k)
\ar@<1ex>[l]^-{\Omega^{\operatorname{\infty}}_{\Gm}}
\ar@<1ex>[d]^N
\\
&
\DAeff(k) 
\ar@<1ex>[u]^K
\ar@<1ex>[r]^-{\Sigma^{\infty}}
&
\DA(k)
\ar@<1ex>[l]^-{\Omega^{\operatorname{\infty}}}
\ar@<1ex>[u]^K
}$
\end{center}
where $\Hp(k)$ is the pointed unstable homotopy category, $\SHS(k)$ (resp. $\SH(k)$) is the category obtain after $S^1$-stabilization (resp. $\PP^1$-stabilization) and $\DAeff(k)$ (resp. $\DA(k)$) the (resp. stable) effective $\AAA^1$-derived category (see \cite[§5]{CD12}); all of these triangulated categories are equipped with Morel's homotopy t-structure. Since $\Omega^{\infty}$ and $K$ are t-exact, we have an equivalence (of additive symmetric monoidal categories) $\DA(k)^{\heartsuit}\simeq \SH(k)^{\heartsuit}$ and an adjunction on the respective hearts (see also \cite[§4]{DegBon17}):
\begin{center}

$\xymatrix{
\HI(k)
\ar@<1ex>[r]^{\sigma^{\infty}}
&
\HM(k).
\ar@<1ex>[l]^{\omega^{\infty}}
}$
\end{center}

\begin{Lem} \label{FullyFaithfullHearts}
With the previous notations, if the functor $\Sigma^{\infty}$ is fully faithful, then so is $\sigma^{\infty}$.
\end{Lem}
\begin{proof}
For any sheaf $M\in \HI(k)$, we have $\omega^{\infty}(M)=\Omega^{\infty}(M)=H_0\Omega^{\infty}(M)$ and $\sigma^{\infty}(M)=\tau_{\geq 0}\Sigma^{\infty}(M)=H_0\Sigma^{\infty}(M)$ hence the arrow
\begin{center}
$\Id \to \omega^{\infty}\sigma^{\infty}=H_0(\Id \to \Omega^{\infty}\Sigma^{\infty})$
\end{center}
is an isomorphism if $\Sigma^{\infty}$ is fully faithful.
\end{proof}

\begin{Rem}
The converse is not true in general (except maybe for connective spectra).
\end{Rem}
We also have the following commutative diagram of categories
\begin{center}
$\xymatrix{
\DAeff(k) 
\ar[d]
\ar[r]^{\Sigma^{\infty}}
&
\DA(k)
\ar[d]
\\
\widetilde{\mathbf{DM}}^{\operatorname{eff}}(k)
\ar[r]^{\Sigma^{\infty}_{\operatorname{MW}}}
&
\widetilde{\mathbf{DM}}(k)
}$
\end{center}
where $\widetilde{\mathbf{DM}}^{\operatorname{eff}}(k)$ and $\widetilde{\mathbf{DM}}(k)$ are the categories of (effective) MW-motivic complexes (see \cite[Chapter 3]{BCDFO}). Looking at the respective hearts, we thus obtain the following commutative diagram
\begin{center}
$\xymatrix{
\HI(k)
\ar@<1ex>[r]^-{\sigma^{\infty}}
\ar@<1ex>[d]^{\tilde{\gamma}^*}
&
\HM(k)
\ar@<1ex>[l]^-{\omega^{\operatorname{\infty}}}
\ar@<1ex>[d]^{\gamma^*}
\\
\HIMW(k)
\ar@<1ex>[r]^-{\sigma^{\operatorname{\infty}}_{\operatorname{MW}}}
\ar@<1ex>[u]^{\tilde{\gamma}_*}
&
\HM^{\operatorname{MW}}(k).
\ar@<1ex>[l]^-{\omega^{\operatorname{\infty}}_{\operatorname{MW}}}
\ar@<1ex>[u]^{\gamma_*}
}$
\end{center}
Finally, we recall the two following well-known theorems in order to motivate Theorem \ref{EssImageTransfers}.
\begin{The}
With the previous notations, the adjunction
\begin{center}

$\xymatrix{
\HM(k)
\ar@<1ex>[r]^-{\gamma^*}
&
\HM^{\operatorname{MW}}(k).
\ar@<1ex>[l]^-{\gamma_*}
}$
\end{center}
is an equivalence of categories.

\end{The}
\begin{proof}
See \cite[Theorem 5.2]{Fel19}.
\end{proof}
\begin{The}
With the previous notations, the functor ${\sigma^{\operatorname{\infty}}_{\operatorname{MW}}}:\HIMW(k)\to \HM^{\operatorname{MW}}(k)$ is fully faithful.
\par In particular, if $M\in \HIMW(k)$ is a sheaf with MW-transfers, then for any natural number $n$, there exists a sheaf with MW-transfers $N\in \HIMW(k)$ such that $M\simeq N_{-n}$.

\end{The}
\begin{proof}
Use Lemma \ref{FullyFaithfullHearts} and \cite[Chapter 3, Cor. 3.3.9]{BCDFO}.
\end{proof}

The previous theorems hint that similar results should hold for the functor ${\tilde{\gamma}_*:\HIMW(k)\to \HI(k)}$ that forgets MW-transfers. This functor is clearly faithful and conservative but cannot be full according to the following counterexample due to Bachmann: 
\par Consider the constant sheaf $\ZZ$, which admits MW-transfers. For any MW-homotopy sheaf $F$, the set of maps $\Hom_{\HIMW}(\ZZ,F)$ injects into the subset of $F(k)$ given by the annihilator of the fundamental ideal $\mathbf{I}$ of $\GW(k)$ acting on $F(k)$ (since $\ZZ = \GW/\mathbf{I}$ and maps with MW-transfers from $\GW$ to $F$ are given by $F(k)$). On the other hand, the set of maps $\Hom_{\HI}(\ZZ,F)$ is given by all of $F(k)$.
\par However, we can still characterize its essential image thanks to the following Theorem.

\begin{The} \label{EssImageTransfers}
Let $M\in \HI(k)$ be a homotopy sheaf. The following assertions are equivalent:
\begin{enumerate}
\item[(i)]  There exists $M'\in \HI(k)$ satisfying Conjecture
 \ref{ConjectureMoinsUn} and such that $M\simeq M'_{-1}$.
\item[(ii)] There exists a structure of generalized transfers on $M$.
\item[(iii)] There exists a structure of MW-transfers on $M$.
\item[(iv)] There exists $M''\in \HI(k)$ such that $M\simeq M''_{-2}$. 
\end{enumerate}
\end{The}
\begin{proof}
\par $(i)\Rightarrow (ii)$ See Theorem \ref{HIMoinsUntogtr}.
\par $(ii)\Rightarrow (iii)$ See Theorem \ref{EquivalenceHIMWandHIgtr}.
\par $(iii)\Rightarrow (iv)$ Assume $M\in \HIMW(k)$. Put $M_*=\sigma^{\infty}_{\operatorname{MW}}(M)\in \HM^{\operatorname{MW}}(k)$ so that we have $\omega^{\infty}_{\operatorname{MW}}(M_*)=M_0\simeq (M_2)_{-2}$. Since $\sigma^{\infty}_{\operatorname{MW}}$ is fully faithful, the map $M\to \omega^{\infty}_{\operatorname{MW}} \sigma^{\infty}_{\operatorname{MW}}(M)$ is an isomorphism thus $M\simeq M''_{-2}$ with $M''=M_2$.
\par $(iv)\Rightarrow (i)$ Straightforward.
\end{proof}

\begin{Rem}
Keeping the previous notations, we remark that the equivalence $(i) \Leftrightarrow (ii)$ was conjectured by Morel in \cite[Remark 5.10]{Mor11}.
\end{Rem}

\subsection{Towards conservativity of $\Gm$-stabilization}
\label{SubsectionTowardsConservativity}
We end with a discussion about a conjecture introduced in \cite{BachmannYakerson18}. 
In the classical theory of topological spaces, the functor $\operatorname{Spc}_*\to \mathfrak{D}(\Ab)$, sending a space to its singular chain complex, is conservative on (at least) simply connected spaces. We would like to study a similar question in the motivic context: up to which extent is the functor
\begin{center}
$\Sigma^{\infty}_{\Gm}:\SHS(k) \to \SH(k)$
\end{center}
 conservative ? The conjecture of Bachmann and Yakerson relied on the hope that this is true after $\Gm$-suspension. Precisely, for any natural number $n$, denote by $\SHS(k)(n)$ the localizing subcategory of $\SHS(k)$ generated by $\Sigma^n_{\Gm}\Sigma^{\infty}_{S^1}X_+$ where $X\in \Sm$. The fact that the functor $\Sigma^{\infty}_{\Gm}:\SHS(k)(1) \to \SH(k)$ is conservative on bounded below objects reduces to proving the following statement (see \cite[Conjecture 1.1]{BachmannYakerson18}). 

\begin{Conj}[Bachmann-Yakerson]
If $d\geq 1$ is a natural number, then the canonical functor
\begin{center}
$\Sigma^{\infty -d\heartsuit}_{\Gm}:\SHS(k)(d)^{\heartsuit} \to \SH(k)^{\txt{eff}\heartsuit}$
\end{center}
is an equivalence of abelian categories.
\end{Conj}
Recall that $\SH(k)^{\txt{eff}}$ denotes the localizing subcategory generated by the image of $\SHS(k)$ in $\SH(k)$ under $\Sigma^{\infty}_{\Gm}$ and the hearts are taken with respect to homotopy t-structures on these categories. As a reformulation of the conjecture, we remark that the category $\SHS(k)^{\heartsuit}$ is equivalent to the category of homotopy sheaves $\HI(k)$ and that the category $\SH(k)^{\txt{eff}\heartsuit}$ is equivalent to the category $\HIfr(k)$ of homotopy sheaves with {\em framed transfers} (see \cite[Theorem 5.14]{BachmannYakerson18}). In \cite{Bach20}, Bachmann proved the following theorem.
\begin{The} \label{BachThm4.5}
Let $d>0$ be a natural number. If $d>1$ or Conjecture \ref{ConjectureMoinsUn} holds, then for any homotopy sheaf $M\in \HI(k)$, the Bass-Tate transfers on $M_{-d}$ extend to framed transfers and the canonical functor
\begin{center}
$\Sigma^{\infty -d\heartsuit}_{\Gm}:\SHS(k)(d)^{\heartsuit} \to \SH(k)^{\txt{eff}\heartsuit}$
\end{center}
is an equivalence of abelian categories.
\end{The}
\begin{proof}
Let $M\in \HI(k)$ be a homotopy sheaf. If $d>1$ (resp. if Conjecture \ref{ConjectureMoinsUn} holds), then $M_{-d}$ (resp. $M_{-1}$) has a structure of generalized transfers hence of Milnor-Witt transfers according to Section \ref{SectionMWHomotopySheaves}. Thus it has a structure of framed transfers (see \cite[Chapter 3, §2]{BCDFO}) and the first result holds. The second one is \cite[Theorem 4.5]{Bach20}.
\end{proof}

As an application of our theorem \ref{HIMoinsUntogtr}, we obtain:
\begin{Cor}\label{RationalBYconjecture} Let $d>0$ be a natural number. The Bachmann-Yakerson conjecture holds (integrally) for $d=2$ and rationally for $d=1$: namely, the canonical functor
 \begin{center}
$\SHS(k)(2) \to \SH(k)$
\end{center}
is conservative on bounded below objects, the canonical functor
 \begin{center}
$\SHS(k)(1) \to \SH(k)$
\end{center}
is conservative on rational bounded below objects, and the canonical functor
\begin{center}
$\HI(k,\QQ)(1)\to \HIfr(k,\QQ)$
\end{center}
is an equivalence of abelian categories.
\par Moreover, let $\mathcal{X}$ be a pointed motivic space. Then the canonical map
\begin{center}
$\underline{\pi}_0\Omega^{d}_{\PP^1}\Sigma^{d}_{\PP^1}\mathcal{X}\to
\underline{\pi}_0\Omega^{d+1}_{\PP^1}\Sigma^{d+1}_{\PP^1}\mathcal{X}$
\end{center}
is an isomorphism for $d=2$.
\end{Cor} 
\begin{proof}
The last result follows as in the proof of \cite[Theorem 1.1]{Bach20}.
\end{proof}Moreover, we can solve (integrally) a question left open after the work of \cite{BCDFO} and \cite{GarkushaPanin14}:
\begin{Cor} \label{EquivalenceHIMWandHIfr}
The category of homotopy sheaves with generalized transfers, the category of MW-homotopy sheaves and the category of homotopy sheaves with framed transfers are equivalent:
\begin{center}
$\HIgtr(k)\simeq \HIMW(k) \simeq \HIfr(k)$.
\end{center}
\end{Cor}
\begin{proof}
The first equivalence is our Theorem \ref{EquivalenceHIMWandHIgtr}; the second one is due to Bachmann (combine \cite[Proposition 29]{Bach18bis} and \cite[Theorem 5.14]{BachmannYakerson18}).
\end{proof}

  \bibliographystyle{alpha}
  \bibliography{BiblioFile3}


\end{document}